\documentclass[amssymb,amsfonts,refcheck,12pt,verbatim,righttag]{amsart}

\usepackage{color}
\usepackage{amssymb}
\usepackage{graphics}
\usepackage{tikz}
\usepackage{multicol}
\usepackage{url}
\usepackage{subfigure}
\usepackage{bm}

 \numberwithin{equation}{section} 
\newcommand{\beq}{\begin{equation}}
\newcommand{\eeq}{\end{equation}}

\newcommand{\ben}{\begin{eqnarray}}
\newcommand{\een}{\end{eqnarray}}

\newcommand{\bet}{\begin{eqnarray*}}
\newcommand{\eet}{\end{eqnarray*}}

\DeclareMathOperator{\ord}{ord}

\newtheorem{thm}{Theorem}[section]
\newtheorem{lem}[thm]{Lemma}

\newtheorem{cor}[thm]{Corollary}
\newtheorem{conj}[thm]{Conjecture}

\newtheorem{specialcon}{Conjecture}
\newtheorem{rem}[thm]{Remark}

\usepackage[colorlinks,citecolor=red,hypertexnames=false]{hyperref}

\newcommand{\R}{\mathbb{R}}

\newcommand{\N}{\mathbb{N}}
\newcommand{\Z}{\mathbb{Z}}

\theoremstyle{plain}

\makeatletter

\newcommand{\Rmnum}[1]{\expandafter\@slowromancap\romannumeral #1@}
\makeatother
\begin{document}
\baselineskip 16pt

\title{Metric results for dyadic approximation on the middle-third Cantor set}

\author{Xin-Rong Dai}
\author{Bing Li} 
\author{Bo Wang}
\author{Yu-Feng Wu}

\address[]{School of Mathematics, Sun Yat-sen University, Cuangzhou, 510275, P. R. China}
\email{daixr@mail.sysu.edu.cn}

\address[]{School of Mathematics, South China University of Technology, Guangzhou,
	510641, P. R. China}
\email{scbingli@scut.edu.cn}

\address[]{School of Mathematics, Sun Yat-sen University, Cuangzhou, 510275, P. R. China \\ School of Mathematics, Jiaying University, Meizhou, 514015, P. R. China}
\email{math\_bocomeon@163.com}

\address[]{School of Mathematics and Statistics\\HNP-LAMA\\Central South University\\ Changsha, 410083, P. R. China}
\email{yufengwu.wu@csu.edu.cn}


\begin{abstract}
Let $C$ be the middle-third Cantor set and $\mu$ be the Cantor-Lebesgue measure on $C$. A conjecture of Velani states that
\[\mu(W_2(\tau))=\begin{cases}
	0 \quad  \text{ if }\tau>1,\\
	1 \quad \text{ if }0<\tau\leq 1,
\end{cases}\]
where 
\[W_2(\tau)=\left\{x\in[0,1]: \|2^nx\|<n^{-\tau}\text{ for infinitely many }n\in \N\right\}.\]
We prove that the conjecture holds for  $\tau>\frac{1}{\gamma}-\frac{1-\gamma}{3-\gamma}\,(\approx 1.429)$
 and $0<\tau<\frac{\gamma}{12}\,(\approx 0.052)$, where $\gamma=\frac{\log2}{\log3}$ is the Hausdorff dimension of $C$. This improves the known results on both the null part ($\tau>\frac{1}{\gamma}-\frac{0.078(1-\gamma)}{\gamma(2-\gamma)}\approx 1.552$, due to Allen, Baker, Chow, and Yu~\cite{ABCY23A}) and the full measure part ($0<\tau\leq 0.01$, due to Baker~\cite{Baker25A}). Our key innovation is to establish the estimate
\[\sum_{n=1}^{N}|\widehat{\mu}(h2^n)|^2\ll N^{1-\gamma}\]
and its consequences: 
\[  \sum_{n=1}^{N}|\widehat{\mu}(h2^n)|\ll N^{1-\frac{\gamma}{2}},\quad  \sum_{n=1}^{N}n^{- \sigma}|\widehat{\mu}(h2^n)|\ll_{\sigma} N^{1-\frac{\gamma}{2}-\sigma},\]
where $0<\sigma<1-\frac{\gamma}{2}$, and all estimates are uniform in $h\in\Z\setminus\{0\}$. For the  full measure part, our approach also generalizes to self-similar measures on a class of missing-digit sets.
\end{abstract}

\maketitle

\section{Introduction}

 The middle-third Cantor set $C$ consists of all real numbers in $[0,1]$ whose ternary expansions  have only  digits $0$ and $2$. That is,
\begin{equation*}
	C:=\left\{\sum_{j=1}^{\infty}\frac{a_j}{3^j}: a_j\in \{0,2\}, j\geq 1\right\}.
\end{equation*}
In 1984,  Mahler \cite{Mahler84} proposed to study the problem of how well  elements in $C$ can  be
approximated by rational numbers in $C$ (intrinsic approximation) and by rational numbers outside  $C$ (extrinsic approximation). Mahler's problem pioneered the study of Diophantine approximation on fractals, leading to a substantial amount of research; see, for instance, \cite{LSV07O,BD16M,Schleischitz21,TWW24M,Baker25A,CVY24C,CU25S,BHZ24K} and references therein. We do not attempt to give an exhaustive review. Instead, we highlight the results most directly relevant to our work.

In~\cite{LSV07O} Levesley, Salp and Velani studied approximation on $C$ by triadic rationals; that is, rationals whose  denominators  are powers of $3$.  More precisely, let $\mu$ be the  Cantor-Lebesgue measure on $C$. Equivalently, $\mu$ is exactly the $\gamma$-dimensional Hausdorff measure restricted to $C$, where 
\[\gamma:=\frac{\log2}{\log3}\]
is the Hausdorff dimension of $C$. Given $\psi: \N\to [0,\infty)$, set 
\[W_3(\psi):=\left\{x\in [0,1]: \|3^nx\|<\psi(n)\text{ for infinitely many   }n\in\N\right\}.\]
Here and throughout,  $\|\cdot\|$ denotes  the Euclidean distance to the nearest integer.   Levesley, Salp and Velani~\cite{LSV07O} proved a Khintchine-type result for $\mu(W_3(\psi))$ stating that $\mu(W_3(\psi))$ is either null or full according to the convergence or divergence of a certain series.  This result was generalized by Li, Li and Wu~\cite{LLW25Z} to the case of multiplicatively dependent bases.

For approximation on $C$ with dyadic rationals, the situation is substantially different and much more difficult. Given $\psi:\N\to [0,\infty)$ let 
\begin{equation}\label{eqW2psi}
W_2(\psi):=\left\{x\in [0,1]: \|2^nx\|< \psi(n)\text{ for infinitely many   }n\in\N\right\}.
\end{equation}
The following conjecture was attributed to Velani in  \cite{ACY24}. 
\begin{conj}
\label{ConjVelani}
Let $\psi:\N\to [0,\infty)$ be monotonic. Then 
\[\mu(W_2(\psi))=\begin{cases}
	0 \quad  \text{ if }\sum_{n=1}^{\infty}\psi(n)<\infty,\\
	1 \quad \text{ if }\sum_{n=1}^{\infty}\psi(n)=\infty.
\end{cases}\]
\end{conj}
For approximation functions of the form $\psi_{\tau}(n)=n^{-\tau}$ ($\tau> 0$), write $W_2(\psi_{\tau})$ as $W_2(\tau)$ for short. In this case,  Velani's conjecture reduces to the following. 
\begin{specialcon}
\label{ConjVelanitau}
We have \[\mu(W_2(\tau))=\begin{cases}
	0 \quad  \text{ if }\tau>1,\\
	1 \quad \text{ if }0<\tau\leq 1.
\end{cases}\]
\end{specialcon}

Although there have been several deep results, Conjectures~\ref{ConjVelani} and \ref{ConjVelanitau} are widely open. Allen, Chow, and Yu \cite{ACY24} made the first progress in both the convergence and divergence parts of Conjecture \ref{ConjVelani}.  They proved that $\mu(W_{2}(\psi))=0$ if 
$$\sum_{n=1}^{\infty}\left(2^{-\log n/\log \log n \cdot \log \log \log n}\psi(n)^{\log 2/\log 3}+\psi(n)\right)<\infty,$$
and $\mu(W_{2}(\psi))=1$ if $\psi(n)=2^{-\log \log n/\log \log \log n}$; see \cite[Theorems 1.5, 1.9]{ACY24}.
A key ingredient in \cite{ACY24} is to relate   Velani's conjecture to the principle that the base-$2$ and base-$3$ expansions of a number are not both structured. However, the current known results supporting this principle are much weaker than conjectured, which prevents the method of \cite{ACY24} from fully resolving Velani's conjecture; see \cite[Section 5]{ACY24} for some conditional results. This indicates the difficulty of the conjecture. 

The above principle  falls within the broader context of  Furstenberg's $\times 2, \times 3$ principle from dynamical systems. Broadly speaking, Furstenberg's $\times 2, \times 3$ principle describes the independence of the maps $T_2: x\mapsto 2x\bmod1$ and $T_3: x\mapsto 3x\bmod1$ on the unit circle  at various levels. This principle has given rise to  many open problems and deep results, see, for instance \cite{Furstenberg67,Host95,HS15E,Shmerkin19,Wu19A,CLW26}.  We refer the reader to \cite[Section 1]{Baker25A} for a more detailed discussion of related  results.

The connection between Furstenberg's $\times 2, \times 3$ principle and Velani's conjecture  was also emphasized by Baker~\cite{Baker25A}.
To  study  this principle in the context of shrinking targets problems, 
Baker~\cite{Baker25A} proposed a  conjecture which generalizes Velani's conjecture to the  broader framework of shrinking targets, without assuming monotonicity of $\psi$. More precisely,
given a sequence  $\mathbf{x}=(x_n)_{n=1}^{\infty}$ in $[0,1]$ and a function $\psi:\N\to[0,\infty)$, set
\[W_2(\psi,\mathbf{x}):=\left\{x\in[0,1]: \|2^nx-x_n\|<\psi(n)\text{ for infinitely many }n\in \N\right\}.\]
Note that when $x_n=0$ for all $n$, $W_2(\psi,\mathbf{x})$ coincides with $W_2(\psi)$ defined in \eqref{eqW2psi}.
Baker conjectured that 
\begin{equation}\label{BakerConj}
\mu(W_2(\psi,\mathbf{x}))=\begin{cases}
	0 \quad  \text{ if }\sum_{n=1}^{\infty}\psi(n)<\infty,\\
	1 \quad \text{ if }\sum_{n=1}^{\infty}\psi(n)=\infty.
\end{cases}
\end{equation}
Regarding this conjecture, Baker~\cite{Baker25A} confirmed the divergent part for approximation functions $\psi_{\tau}(n)=n^{-\tau}$ with $0<\tau\leq 0.01$. Indeed, he proved a more general counting result stating that for $\mu$-a.e. $x\in C$, 
\begin{equation}\label{BakerMainthm}
\lim_{N\to\infty}\frac{\#\left\{1\leq n\leq  N:\|2^nx-x_n\|<n^{-0.01}\right\}}{2\sum_{n=1}^{N}n^{-0.01}}=1.
\end{equation}
In particular, this implies that
\begin{equation}\label{eqBaker001}
\mu(W_2(\tau))=1, \qquad \forall 0<\tau\leq 0.01,
\end{equation}
which represents a step towards the full measure part of Velani's Conjecture~\ref{ConjVelanitau}.

By applying a key lemma from \cite{Baker25A} (cf. Lemma~\ref{Bakerlemma})  and some other techniques,  
Allen, Baker, Chow, and Yu \cite{ABCY23A} made further progress on the null part of Conjecture~\ref{ConjVelanitau}. They considered the more general inhomogeneous setting. More precisely, let $y\in \R$ be fixed. 
For $\tau>0$, denote $\psi_{\tau}(n)=n^{-\tau}$ and
\begin{equation}\label{eqW2tauy}
 W_{2}(\psi_{\tau},y):=\left\{x\in[0,1]: \|2^{n}x-y\|<n^{-\tau}\text{ for infinitely many } n\in\mathbb{N}\right\}. 
\end{equation}
Allen, Baker, Chow, and Yu~\cite{ABCY23A} proved that 
\begin{equation}\label{ThmABCYCon}
\mu(W_2(\psi_{\tau}, y)) = 0, \qquad  \forall \tau>\frac{1}{\gamma}-\frac{0.078(1-\gamma)}{\gamma(2-\gamma)}.
\end{equation}

In this paper, we  further explore Velani's Conjecture~\ref{ConjVelanitau}.  Our first main result is the following theorem, which improves the known results on both the null and  full measure parts on Conjecture~\ref{ConjVelanitau}. 

\begin{thm}\label{MainThm0}
We have 
\[\mu(W_2(\tau))=\begin{cases}
	0 \quad  \text{ if }\tau>\frac{1}{\gamma}-\frac{1-\gamma}{3-\gamma},\\
	1 \quad \text{ if }0<\tau<\frac{\gamma}{12}.
\end{cases}\]
\end{thm}
\begin{rem}
Note that \[\frac{\gamma}{12}\approx 0.052, \quad \frac{1}{\gamma}-\frac{1-\gamma}{3-\gamma}\approx 1.429 \quad \text{and} \quad \frac{1}{\gamma}-\frac{0.078(1-\gamma)}{\gamma(2-\gamma)} \approx1.552.\]
Hence Theorem~\ref{MainThm0} improves both Baker's result \eqref{eqBaker001} and the result \eqref{ThmABCYCon} due to Allen, Baker, Chow and Yu. 
\end{rem}

For the full measure part of Theorem~\ref{MainThm0}, similar to \cite{Baker25A}, we indeed prove the following more general counting result, which generalizes Baker's result \eqref{BakerMainthm}.

\begin{thm}\label{Mainthm}
Let $(x_{n})_{n=1}^{\infty}$ be a sequence of real numbers in $[0,1]$. Let $0<\tau<\frac{\gamma}{12}$. Then for $\mu$-a.e. $x\in C$, we have  $$\lim_{N\to\infty}\frac{\#\left\{1\leq n\leq  N:\|2^nx-x_n\|<n^{-\tau}\right\}}{2\sum_{n=1}^{N}n^{-\tau}}=1.$$ 
\end{thm}

Let $W_{2}(\psi_{\tau},y)$ be defined as in \eqref{eqW2tauy}. The null part of Theorem~\ref{MainThm0} follows from the following. 

\begin{thm}\label{thmconverge}
Let \(\tau>\frac{1}{\gamma}-\frac{1-\gamma}{3-\gamma}\). Then
	\(\mu (W_{2}(\psi_{\tau},y)) = 0\).
\end{thm}
\begin{rem}  
 The proof of Theorem~\ref{thmconverge} still works if $y$ is replaced by a sequence $(x_{n})_{n=1}^{\infty}$ in $[0,1]$.
	That is, we have $\mu(W_{2}(\psi_{\tau},\mathbf{x}))=0$ for any $\tau>\frac{1}{\gamma}-\frac{1-\gamma}{3-\gamma}$. Therefore, Theorems~\ref{Mainthm} and \ref{thmconverge} also make progress towards Baker' conjecture~\eqref{BakerConj}. 
\end{rem}

Finally, we remark that our approach to prove Theorem~\ref{Mainthm} can also be extended to  give a counting result with respect to general  self-similar measures for approximation on some missing-digit sets. More exactly, given an integer $b\geq 3$ and a set  $D\subseteq\{0,1,\ldots,b-1\}$ with cardinality  $2\leq \#D<b$, the {\em missing-digit set} with base $b$ and digit set $D$, denoted by $K_{b,D}$, is defined as
\[K_{b,D}:=\left\{\sum_{j=1}^{\infty}\frac{\varepsilon_j}{b^j}: \varepsilon_j\in D, j\geq 1\right\}.\]
Clearly, $K_{3,\{0,2\}}$ is the middle-third Cantor set $C$. Let $\mathbf{p}=(p_{1},p_{2},\ldots,p_{\#D})$ be a probability vector with strictly positive entries, i.e., $p_{1},p_{2},\ldots,p_{\#D}\in(0,1)$ and $\sum_{i=1}^{\#D}p_{i}=1$. Let $\mu_{\mathbf{p}}$ be the self-similar measure supported on $K_{b,D}$ associated with probability vector $\textbf{p}$. That is, $\mu_{\textbf{p}}$ is the unique Borel probability measure supported on $K_{b,D}$ satisfying 
\[\mu_{\mathbf{p}}=\sum_{j=1}^{\#D}p_j\mu_{\mathbf{p}}\circ\phi_j^{-1},\]
where 
\[\phi_j(x)=\frac{x+k_j}{b}, \quad D=\{k_1,k_2,\ldots, k_{\#D}\}\ {\rm and}\  k_1<k_2<\cdots<k_{\#D}.\]

We have the following extension of Theorem~\ref{Mainthm}.
\begin{thm}\label{thmsselfsimilar}
	Let $(x_{n})_{n=1}^{\infty}$ be a sequence of real numbers in $[0,1]$. Suppose $b\geq3$ is a prime number and $t\geq2$ is an integer satisfying that $b\nmid t$. Let  $$\kappa :=\frac{-\log\left(\min\limits_{1\leq j_{1}<j_{2}\leq\#D}(1-2p_{j_{1}}p_{j_{2}})\right)}{12\log b}.$$ Then for any $\tau\in(0,\kappa)$, we have for $\mu_{\textbf{p}}$-a.e. $x\in K_{b,D}$,
	$$\lim_{N\to\infty}\frac{\#\{1\leq n\leq N:\|t^{n}x-x_{n}\|<n^{-\tau}\}}{2\sum_{n=1}^{N}n^{-\tau}}=1.$$
\end{thm}

\begin{rem}
Note that when $K_{b,D}$ is the middle-third Cantor set and $\mu_{\textbf{p}}$ is the Cantor-Lebesgue measure, $\kappa$ in Theorem~\ref{thmsselfsimilar} coincides with the threshold $\frac{\gamma}{12}$ in Theorem~\ref{Mainthm}. 
\end{rem}

At the end of this section, we briefly describe the  key ideas  in our proofs of the main results. 
Similar to previous works \cite{ABCY23A,ACY24,Baker25A}, our method is Fourier-analytic. A main obstruction is that $\mu$ has no global Fourier decay (i.e., $|\widehat{\mu}(\xi)|\not\to 0$ as $|\xi|\to\infty$). To get around this, Baker~\cite[Lemma 2.2]{Baker25A} proved that for a fixed non-zero integer $h$, $\widehat{\mu}(h 2^n)$ decays to zero polynomially fast outside a relatively small set of $n$'s; see also   Lemma~\ref{Bakerlemma}. This result underpins the development in~\cite{Baker25A} and is also  crucial in~\cite{ABCY23A}. In this paper, our key new ingredient  is establishing the following estimates 
$$\sum_{n=1}^{N}|\widehat{\mu}(h2^n)|^2\ll N^{1-\gamma}, \quad  \sum_{n=1}^{N}|\widehat{\mu}(h2^n)|\ll N^{1-\frac{\gamma}{2}},\quad \sum_{n=1}^{N}n^{-\sigma}|\widehat{\mu}(h2^n)|\ll_{\sigma} N^{1-\frac{\gamma}{2}-\sigma},$$
which are uniform in $h\in \Z\setminus\{0\}$; see Lemmas~\ref{LemL2-dyadic} and \ref{lem:weightedL1}. These estimates refine Baker's lemma and  are of importance in their own right. Then our general strategy is to take these as the Fourier input in  the mechanism in~\cite{Baker25A} and~\cite{ABCY23A} to establish Theorems~\ref{Mainthm} and~\ref{thmconverge}. Certainly, several other techniques are  also necessary and will be apparent in the course of the proofs. 

Since the proof of Theorem~\ref{thmsselfsimilar} is analogous to that of Theorem~\ref{Mainthm}, we give the full details only for the latter. For Theorem~\ref{thmsselfsimilar}, we shall merely indicate the necessary modifications to the key lemmas.

The paper is organized as follows. In Section~\ref{S2Preliminary}, we establish several estimates on the Fourier transform of $\mu$; the main objects are Lemmas~\ref{LemL2-dyadic} and \ref{lem:weightedL1}. Then in Section~\ref{S3}, we apply these estimates (especially Lemma~\ref{lem:weightedL1}) to derive integral estimates by following the mechanism of  \cite{Baker25A}. We then prove Theorem~\ref{Mainthm} in Section~\ref{Sec:Mainthm}. Section~\ref{sec:thmconverge} is devoted to the proof of Theorem~\ref{thmconverge}. Finally, in the last section, we present necessary modifications to the key lemmas needed to prove Theorem~\ref{thmsselfsimilar}.

{\bf Notation.} For complex-valued functions $f$ and $g$, we write $f\ll g$ or $f=O(g)$ if there exist a constant $M>0$ such $|f|\leq M|g|$ pointwise. Throughout, $\Z$ and $\N$ denote respectively the set of  integers and the set of positive integers. 
 
\section{Preliminary lemmas}\label{S2Preliminary}

Given a Borel probability measure $\nu$ on $\R$, the Fourier transform of $\nu$ is defined as 
\[\widehat\nu(\xi)=\int e^{-2\pi i \xi y}\,\mathrm{d}\nu(y).\]
Recall that $\mu$ is the Cantor-Lebesgue measure on $C$. Equivalently, $\mu$ is  the self-similar measure on $C$ defined by
\[
\mu = \frac12 \mu\circ S_1^{-1} + \frac12 \mu\circ S_2^{-1},
\qquad \text{where }
S_1(x)=\frac{x}{3},\, S_2(x)=\frac{x+2}{3}.
\]
It is well-known that  (see \cite[page 342]{Strichartz1993})
\begin{equation*}\label{eq:fourier-product}
	\widehat\mu(x)
	= \prod_{q=1}^{\infty} \frac{1+e(-2x/3^q)}{2},
	\qquad
	|\widehat\mu(x)|
	= \prod_{q=1}^{\infty}\left|\cos\left(\frac{2\pi x}{3^q}\right)\right|.
\end{equation*}
Here and throughout, for convenience, we write 
\[e(x)=e^{2\pi i x}.\]
For more information on  self-similar measures,  we refer the reader to \cite{Falconer14}. 

In \cite{Baker25A}, Baker proved  Theorem~\ref{BakerMainthm} based on the following result, which is also a key ingredient in \cite{ABCY23A}. 
\begin{lem}\cite[Lemma 2.2]{Baker25A}\label{Bakerlemma}
	Let $N\in\mathbb{N}$ and $h\in\mathbb{Z}\setminus \{0\}$. Then there exist $C_{1},C_{2}>0$ independent of $N$ and $h$ such that 
	$$\# \left\{0\leq n<N:\left|\widehat{\mu}(h 2^n)\right|> C_{1}N^{-0.078}\right\}\leq C_{2}N^{0.922}.$$
\end{lem}
Lemma~\ref{Bakerlemma} establishes  that,  despite  $\mu$ lacking global Fourier decay, $\widehat{\mu}(h 2^n)$ decays polynomially to zero for all but a  relatively small set of $n$. In this section, instead,  we directly estimate the partial sums of the Fourier coefficients of $\mu$ along the sequence $h2^n$ ($n\in\N$) for any fixed $h\in\Z\setminus\{0\}$. 
The main results are the following two lemmas, which are crucial in the subsequent  development of this paper. Recall that
\[\gamma=\frac{\log2}{\log3},\]
which is the Hausdorff dimension of $C$. 
\begin{lem}\label{LemL2-dyadic}
	There exists an absolute constant $c>0$ such that, for every non-zero integer $h$, every integer $a\ge 0$, and every integer $N\ge 1$,
\begin{equation}\label{eq:L2-main}
		\sum_{n=a}^{a+N-1}|\widehat\mu(h2^n)|^2
		\leq c^2 N^{1-\gamma}.
	\end{equation}
	Consequently, 
\begin{equation}\label{eq:L1-main}
		\sum_{n=a}^{a+N-1}|\widehat\mu(h2^n)|
		\leq c N^{1-\frac{\gamma}{2}}.
	\end{equation}
\end{lem}
In Lemma~\ref{LemL2-dyadic},
the case where $a=0$ is of particular interest, and we state it below.
\begin{cor}\label{cor:prefix}
	For any $h\in\Z\setminus\{0\}$ and  $N\in \N$, we have 
	\begin{equation*}\label{eqsumestimate}
		\sum_{n=0}^{N-1}|\widehat\mu(h2^n)|^2
		\ll N^{1-\gamma} \quad \text{ and }\quad 
		\sum_{n=0}^{N-1}|\widehat\mu(h2^n)|
		\ll N^{1-\frac{\gamma}{2}},
	\end{equation*}
	where the underlying constants are independent of $N$ and  $h$. 
\end{cor}

A direct consequence of Corollary~\ref{cor:prefix} is the following result, which is particularly important in proving Theorem~\ref{Mainthm}; see the proofs of Lemmas~\ref{Lemma1} and~\ref{Lemma2} in Section~\ref{S3}.

\begin{lem}\label{lem:weightedL1}
	Let  $0< \sigma<1-\frac{\gamma}{2}$.  Then for any $h\in\Z\setminus\{0\}$ and  $N\in\N$, we have
\begin{equation*}\label{eq:weighted-L1}
		\sum_{1\leq n\leq N}n^{- \sigma}|\widehat\mu(h2^n)|
		\ll_{\sigma} N^{1-\frac{\gamma}{2}-\sigma},
	\end{equation*}
	where the implicit constant is independent of $N$ and $h$. 
\end{lem}

\begin{proof}
Fix $h\in\Z\setminus\{0\}$. For $X\geq1$, let
\[A_h(X):=\sum_{n\in\N: 1\leq n\leq X}|\widehat\mu(h2^n)|.\]
Then by  Corollary~\ref{cor:prefix}, $A_h(X)\ll X^{1-\frac{\gamma}{2}}$. Moreover,  Abel’s Summation Formula (cf. \cite[Theorem 4.10]{Pongsriiam}) gives
\begin{align*}
		\sum_{1\leq n\leq N}n^{-\sigma}|\widehat\mu(h2^n)|
		&=N^{-\sigma}A_h(N)
		+\sigma\int_1^N x^{-\sigma-1}A_h(x)\, \mathrm{d}x \\
		&\ll_{\sigma} N^{1-\frac{\gamma}{2}-\sigma}
		+\int_1^N x^{-\frac{\gamma}{2}-\sigma}\, \mathrm{d}x\\
		&\ll_{\sigma} N^{1-\frac{\gamma}{2}-\sigma}.
\end{align*}
This proves the lemma.
\end{proof}

It remains to prove Lemma~\ref{LemL2-dyadic}.
For this purpose, we  establish several auxiliary lemmas. 

We first give some notation.  For $m\in\Z\setminus\{0\}$,  $\nu_3(m)$ denotes the {\it $3$-adic valuation} of $m$, i.e., $\nu_3(m)$ is the largest integer $n\geq 0$ such that $3^n\mid m$.  For an integer $q\geq 2$ and an integer $a$ coprime to $q$, let $\ord_q(a)$
be the {\it multiplicative order} of $a$ modulo $q$. That is, $\ord_{q}(a)$ is the smallest $k\in\N$ such that 
\[a^{k}\equiv 1\bmod q.\] The following number-theoretic lemma is due to Cassels~\cite[Lemma 2]{Cassels59}, which we include a proof for completeness.

\begin{lem}\label{lemmod3lr}
Let $b\ge 2$ satisfy $b\equiv 1\bmod{3}$. Set $\ell := \nu_3(b-1)$.
	Then, for every integer $r\ge 0$,
\begin{equation}\label{eqbn1to3r}
		\{b^n \bmod 3^{\ell+r}: 0\leq n<3^r\}
		=
		\{(1+3^{\ell}k)\bmod 3^{\ell+r}: 0\leq k<3^{r}\}.
	\end{equation}
	Equivalently, the powers $b^n$, $0\leq n<3^r$, run exactly once through all residue classes modulo $3^{\ell+r}$ that are congruent to $1$ modulo $3^\ell$.
\end{lem}

\begin{proof}
	By the Lifting the Exponent Lemma \cite[Theorem 1.37]{Pongsriiam},
	\begin{equation}\label{eq:lifting-identity}
		\nu_3(b^n-1)=\ell+\nu_3(n),
		\qquad n\ge 1.
	\end{equation}
	Letting $n=3^r$ in \eqref{eq:lifting-identity},  we obtain
	\[
	b^{3^r}\equiv 1 \bmod{3^{\ell+r}}.
	\]
	If $0<n<3^r$, then $\nu_3(n)\leq r-1$. Hence
	\[
	\nu_3(b^n-1)\leq \ell+r-1,\]
	and so $b^n\not\equiv 1\bmod {3^{\ell+r}}$ for $0<n<3^r$.  Therefore,
\begin{equation*}
		\ord_{3^{\ell+r}}(b)=3^r.
	\end{equation*}
	The $3^r$ residues $b^0,b^1,\ldots,b^{3^r-1}$ are therefore distinct modulo $3^{\ell+r}$.  They all lie in the set of residue classes congruent to $1$ modulo $3^\ell$, and that set has exactly $3^r$ elements.  Hence the two sets in \eqref{eqbn1to3r} coincide.
\end{proof}

Next, we extend  Lemma~\ref{lemmod3lr} to non-zero multiples of $b^n$. 

\begin{cor}\label{corhbnmod}
	Let $b,\ell$ be as in Lemma \ref{lemmod3lr}. Let $h\in\Z\setminus\{0\}$ and write $
	h=3^m u$, where $m\geq 0$ and $u\in\Z\setminus\{0\}$ with $ \gcd(u,3)=1$. Then, for every integer $r\ge 0$,
\begin{equation}\label{eq:h-traversal}
		\{h b^n\bmod{3^{m+\ell+r}}: 0\leq n<3^r\}
		=
		\{(h+3^{m+\ell}k)\bmod{3^{m+\ell+r}}:0\leq k<3^r\}.
	\end{equation}
	\end{cor}

\begin{proof} 
	By Lemma~\ref{lemmod3lr}, we have
	\[
	b^n\equiv (1+3^\ell t_n) \bmod {3^{\ell+r}},
	\]
where $t_n$ ($0\leq n<3^r$) run through all residue classes modulo $3^r$. 
	Multiplying by $h=3^m u$ gives
	\[
	h b^n \equiv (h + 3^{m+\ell}ut_n)
	\bmod{3^{m+\ell+r}}.
	\]
	Since $\gcd(u,3)=1$,  the values $ut_n$ ($0\leq n<3^r$) run through all residue classes modulo $3^r$.  This proves \eqref{eq:h-traversal}.  
\end{proof}

\begin{rem}\label{remb4case}
	For $b=4$, we  have $\ell=\nu_3(4-1)=1$.  Therefore, for any $h\in\Z\setminus\{0\}$ with $\gcd(h,3)=1$,
	\begin{equation*}
		\{h4^n \bmod 3^{r+1}:0\leq n<3^r\}
		=
		\{a\bmod{3^{r+1}}:a\equiv h\bmod{3}\}.
	\end{equation*}
\end{rem}

We also need the following elementary lemma.

\begin{lem}\label{lemaverageCOS2}
	Let $u\in\Z$ with $\gcd(u,3)=1$. Then  for every integer $r\ge 0$,
	\begin{equation*}
		\frac{1}{3^r}\sum_{\substack{y\bmod 3^{r+1}\\ y\equiv u\bmod 3}}
		\prod_{q=1}^{r+1}
		\cos^2\left(\frac{\pi y}{3^q}\right)
		=2^{-r-2}.
	\end{equation*}
\end{lem}

\begin{proof}
The case \(r=0\) is immediate, since for $u\in \{1,2\}$, 
\begin{equation}\label{eqcos222}
\cos^2\left(\frac{\pi u}{3}\right)=\frac14=2^{-2}.
\end{equation}
Assume now that \(r\ge 1\).

By \eqref{eqcos222} we also have 
\[\prod_{q=1}^{r+1}
\cos^2\left(\frac{\pi y}{3^q}\right)=2^{-2}\prod_{q=2}^{r+1}
\cos^2\left(\frac{\pi y}{3^q}\right).\]
Write
	\[
	y=u+3s,
	\qquad 0\leq s<3^r.
	\]
Then it suffices to show that  
\begin{equation}\label{eq13fs0}
\frac{1}{3^r}
\sum_{s=0}^{3^r-1}
\prod_{q=2}^{r+1}
\cos^2\left(\frac{\pi (u+3s)}{3^q}\right)
=2^{-r}.
\end{equation}
Using the identity 
\begin{equation*}
		\cos^2(\pi x)=\frac12+\frac14 e(x)+\frac14 e(-x),
\end{equation*}
	we expand
	\[
	\prod_{q=2}^{r+1}
	\cos^2\left(\frac{\pi y}{3^q}\right).\]
	A typical term is of the form
	\[
	c_{\boldsymbol\varepsilon}
	e\left(y\sum_{q=2}^{r+1}\frac{\varepsilon_q}{3^q}\right),
	\qquad
	\varepsilon_q\in\{-1,0,1\}.
	\]
After substituting \(y=u+3s\), its average over \(s\) is equal to $	c_{\boldsymbol\varepsilon}e\left(u\sum_{q=2}^{r+1}\frac{\varepsilon_q}{3^q}\right)$ times 
\begin{equation}\label{eqaverage3r}
\frac{1}{3^r}\sum_{s=0}^{3^r-1}
e\left(s\sum_{q=2}^{r+1}\varepsilon_q 3^{1-q}\right).
\end{equation}
Put
\begin{equation}\label{Bternary}
B=\sum_{q=2}^{r+1}\varepsilon_q 3^{r+1-q}.
\end{equation}
Then \eqref{eqaverage3r} is rewritten as 
	\[
	\frac{1}{3^r}\sum_{s=0}^{3^r-1} e\left(\frac{sB}{3^r}\right),
	\]
	which equals $1$ when $B\equiv 0\bmod {3^r}$ and vanishes  otherwise.  But
	\[
	|B|\leq 1+3+\cdots+3^{r-1}=\frac{3^r-1}{2}<3^r.
	\]
	Thus $B\equiv 0\bmod {3^r}$ implies $B=0$. This in turn forces that
\begin{equation}\label{eqepss0}
\varepsilon_2=\varepsilon_3=\cdots=\varepsilon_{r+1}=0.
\end{equation}
Indeed, assume $B=0$. Then modulo \(3\) in \eqref{Bternary}, we get
\(\varepsilon_{r+1}\equiv 0\bmod 3\). Since
\(\varepsilon_{r+1}\in\{-1,0,1\}\), this gives
\(\varepsilon_{r+1}=0\). Dividing by \(3\) and repeating the argument,
we obtain \eqref{eqepss0}. 
Therefore, 
\[\frac{1}{3^r}
\sum_{s=0}^{3^r-1}
\prod_{q=2}^{r+1}
\cos^2\left(\frac{\pi (u+3s)}{3^q}\right)
=\frac{1}{3^r}
\sum_{s=0}^{3^r-1}
\prod_{q=2}^{r+1}\frac{1}{2}=2^{-r}.\]
This proves \eqref{eq13fs0} and thus completes the proof of the lemma. 
\end{proof}

Now we are ready to give the proof of Lemma~\ref{LemL2-dyadic}.

\begin{proof}[Proof of Lemma~\ref{LemL2-dyadic}]
For convenience,  define
\begin{equation*}
	P(x) := \prod_{j=1}^{\infty}
	\left|\cos\left(\frac{\pi x}{3^j}\right)\right|.
\end{equation*}
Then
\begin{equation*}
	|\widehat\mu(x)| = P(2x).
\end{equation*}	
	It is enough to prove the corresponding estimates for $P(h2^n)$.
	
We first show that  for every $h\in \Z\setminus\{0\}$ and every $r\ge 0$,
	\begin{equation}\label{eq:4-block}
		\sum_{n=0}^{3^r-1} P(h4^{n})^2
		\leq 3^r2^{-r-2}.
	\end{equation}
	Write
	\[
	h=3^m u,\] 
	where  $m\geq0, u\in \Z\setminus\{0\}$ with  $\gcd(u,3)=1$.
	By the definition of $P$,
	\[
	P(h4^{n})^2
	= \prod_{j=1}^{\infty}
	\cos^2\left(\frac{\pi 3^m u4^{n}}{3^j}\right).
	\]
	The factors with $j\leq m$ are equal to $1$, so after relabelling $q=j-m$ we get
	\[
	P(h4^{n})^2
	= \prod_{q=1}^{\infty}
	\cos^2\left(\frac{\pi u4^{n}}{3^q}\right).
	\]
	Since every factor is at most $1$, we have 
\begin{equation*}
		P(h4^{n})^2
		\le
		\prod_{q=1}^{r+1}
		\cos^2\left(\frac{\pi u4^{n}}{3^q}\right).
	\end{equation*}
Applying Lemma~\ref{lemmod3lr} to $b=4$ (see Remark~\ref{remb4case}), the residues
	\[
	u4^{n}\bmod{3^{r+1}},
	\qquad 0\leq n<3^r,
	\]
	run exactly once through all residue classes $y\bmod{3^{r+1}}$ with $y\equiv u\bmod 3$.  Combining this with Lemma~\ref{lemmod3lr} gives
\[\sum_{n=0}^{3^r-1} P(h4^{n})^2
		\leq
		\sum_{\substack{y\bmod 3^{r+1}\\ y\equiv u\bmod 3}}
		\prod_{q=1}^{r+1}
		\cos^2\left(\frac{\pi y}{3^q}\right)=3^r2^{-r-2}.\]
	This proves \eqref{eq:4-block}.
	
For $N\in\N$, set 
	\[
	S_4(a,N;h):=\sum_{n=a}^{a+N-1} P(h4^n)^2.
	\]
Let $R\geq 0$ be the integer such that 
\[3^R\leq N <3^{R+1}.\]
 Then by \eqref{eq:4-block} (in which we take $h$ to be $h4^{a}$),  we have
	\[
	S_4(a,N;h)
	\leq \sum_{n=0}^{3^{R+1}-1}P(h4^{a}\cdot4^{n})^2\leq 3^{R+1}2^{-(R+1)-2}
	\ll \left(\frac32\right)^R.
	\]
	Since $3^{R}\leq N<3^{R+1}$,
	\[
	\left(\frac32\right)^R
	=(3^R)^{1-\log2/\log3}
	\leq N^{1-\gamma}.
	\]
	Thus
	\begin{equation}\label{eq:4-arbitrary}
		\sum_{n=a}^{a+N-1}P(h4^n)^2\ll N^{1-\gamma} .
	\end{equation}

Next we deduce from \eqref{eq:4-arbitrary} the desired estimate \eqref{eq:L2-main} by splitting $[a,a+N)\cap\mathbb{Z}$ into even and odd $n$'s.   Since 
	\[
	P(h2^n)^2=\begin{cases}P(h4^k)^2 & \text{ if } n=2k,\\
	P(2h4^k)^2 &\text{ if } n=2k+1,
	\end{cases}
	\]
 applying \eqref{eq:4-arbitrary} to $h$ and to $2h$ gives
\begin{equation*}
\sum_{n=a}^{a+N-1}P(h2^n)^2\ll N^{1-\gamma}.
\end{equation*}
This proves \eqref{eq:L2-main}.
	
Finally, by Cauchy-Schwarz inequality, we have 
	\[
	\begin{aligned}
		\sum_{n=a}^{a+N-1}|\widehat\mu(h2^n)|
		&\leq N^{1/2}
		\left(\sum_{n=a}^{a+N-1}|\widehat\mu(h2^n)|^2\right)^{1/2} \\
		&\ll N^{(1+1-\gamma)/2}=N^{1-\frac{\gamma}{2}}.
	\end{aligned}
	\]
	This proves \eqref{eq:L1-main} and thus completes the proof of the lemma.
\end{proof}

\section{Smoothing and Integral estimates}\label{S3}

In this section, based on the estimate established in 
Lemma~\ref{lem:weightedL1},  we follow the method of \cite[Section 2.2]{Baker25A} to  derive integral estimates for smooth approximations of the counting function appearing in Theorem~\ref{Mainthm}.  

For a $\mathbb{Z}$ periodic function $f:\mathbb{R}\to\mathbb{R}$ with $\int_{0}^{1}|f(x)|\, \mathrm{d} x<\infty$, the Fourier coefficients of $f$ are given by 
\begin{equation*}
	\label{Fouriertransform2}
	\widehat{f}(k):=\int_{0}^{1}f(x)e^{-2\pi i k x}\, \mathrm{d} x,  \qquad k\in\Z.
\end{equation*}  
For bounded $f$, it's modulus of continuity is defined as
$$\omega_{f}(x):=\sup\{|f(x_{1})-f(x_{2})|:|x_{1}-x_{2}|\leq x\},\quad x\geq0.$$
Without causing confusion, we just write $\omega_{f}(x)$ as $\omega(x)$ to ease notation.

Similar to \cite{Baker25A}, we will make use of the following Jackson's inequality. 

\begin{lem}\cite[Theorem \Rmnum{4}, Chapter 1]{Jackson94T}\label{Jacksoninq}
	Let \( f : \mathbb{R} \to \mathbb{R} \) be a \( \mathbb{Z} \) periodic \( C^1 \) function such that $f'$ has a modulus of continuity \( \omega \). Then for any $N\in\N$, 
	\[
	\left| f(x) - \sum_{k=-N}^N \widehat{f}(k)e^{2\pi ik x} \right| \ll  \frac{\log N}{N} \omega(1/N), \quad \forall x\in \R,
	\]
 where the implicit constant is independent of \( f \) or \( N \).
\end{lem}

Throughout this section, fix  $$0<\tau\leq \frac{\log2}{6\log3}.$$ 
In the following, 
we  apply the technique in \cite{Baker25A} to construct smooth approximations to the indicator  functions for the sets 
\[\{x\in \R: \|2^{n}x-x_n\|<n^{-\tau}\}, \quad n\in\N.\]
To this end, let $(x_{n})_{n=1}^{\infty}$ be a sequence of real numbers in $[0,1]$. Given constants $c_{1},c_{2}>0$ with $c_{1}<c_{2}$, we let $(f_n)_{n=1}^{\infty}$ be a sequence of $C^2$ real-valued functions on $\R$ satisfying the following properties:
\begin{enumerate}
	\item Each $f_n$ is $\mathbb{Z}$ periodic.
	\item $0\leq f_{n}(x)\leq 1$ for every $x\in\mathbb{R}$ and $n\in \mathbb{N}$.
	\item $f_{n}(x)=1$ for $x\in [x_n-c_{1}n^{-\tau},x_n+c_{1}n^{-\tau}]+\mathbb{Z}.$
	\item $f_{n}(x)=0$ for $x\notin [x_n-c_{2}n^{-\tau},x_n+c_{2}n^{-\tau}]+\mathbb{Z}.$
	\item $\|f'_{n}\|_{\infty}\ll n^{\tau}.$
	\item $\|f''_{n}\|_{\infty}\ll n^{2\tau}.$
\end{enumerate}
Here $\|\cdot\|_{\infty}$ denotes the supremum norm. We emphasise that the underlying constants appearing in (5) and (6) do depend upon $c_{1}$ and $c_{2}$, but this will not affect the result so we suppress it from our notation. 

By the property (6) and the mean value theorem, $f'_n$ has a  modulus of continuity $\omega$ such that $\omega(x)\ll n^{2\tau}|x|$. Then it follows from Lemma~\ref{Jacksoninq} that for any $N\in\N$, we have 
\begin{equation}
	\label{errorbound}
	\sum_{\ell\notin [-N,N]}c_{\ell,n}e(\ell x)=O\left(\frac{n^{2\tau}\log N}{N^{2}}\right).
\end{equation}
Here and afterwards, for each $n\in\mathbb{N}$, we let $(c_{\ell,n})_{\ell \in\mathbb{Z}}$ denote the Fourier coefficients of $f_n$. Note that for any $n\in\N$, 
\begin{equation}\label{cellnasy}
|c_{0,n}|\asymp n^{-\tau}, \quad |c_{\ell,n}|\ll n^{-\tau},\quad  \forall \ell\in \Z. 
\end{equation}

Our purpose in this section  is to  prove the following two lemmas, which respectively strengthen Lemmas 2.3 and 2.4 in \cite{Baker25A}.

\begin{lem}
	\label{Lemma1}
	Let $(f_n)_{n=1}^{\infty}$ be a sequence of functions satisfying properties (1)-(6) for some $c_{1},c_{2}$ and $(x_{n})_{n=1}^{\infty}$. Then for any $N\in\mathbb{N}$ we have 
\begin{equation}\label{eqsumfn1N}
\int \sum_{n=1}^{N}f_{n}(2^nx)\, \mathrm{d}\mu = \sum_{n=1}^{N}c_{0,n}+O\left(N^{1-\frac{\log2}{3\log3}}\log N\right)
\end{equation}
 and $$\int \sum_{n=1}^{N}f_{n}(2^nx)\, \mathrm{d}\mu \asymp N^{1-\tau}.$$
\end{lem}

\begin{lem}
	\label{Lemma2}
	Let $(f_n)_{n=1}^{\infty}$ be a sequence of functions satisfying properties (1)-(6) for some $c_{1},c_{2}$ and $(x_{n})_{n=1}^{\infty}$. For any $N\in\mathbb{N}$ we have 
	$$\int \left(\sum_{n=1}^{N}f_{n}(2^nx)-\int \sum_{n=1}^{N}f_{n}(2^ny)\, \mathrm{d}\mu(y) \right)^2\,  \mathrm{d}\mu(x)=O\left(N^{2-\frac{\log2}{6\log3}}\log N\right).$$
\end{lem}
We prove Lemmas \ref{Lemma1} and \ref{Lemma2} by following the strategy in \cite{Baker25A}. The key differences is that we invoke Lemmas \ref{lem:weightedL1} rather than Baker's Lemma ( see Lemma \ref{Bakerlemma}) as the Fourier input. Since the require modifications are substantial, we provide detailed proofs for the sake of completeness and clarity.

\begin{proof}[Proof of Lemma \ref{Lemma1}] 
Put
	\begin{equation*}
\rho=\tau+\frac{\log2}{6\log3}.
	\end{equation*}
	Fix $N\in\mathbb{N}$. By \eqref{errorbound} and the fact that each $f_n$ coincides with its Fourier series, we have  $$\int \sum_{n=1}^{N}f_{n}(2^nx)\, \mathrm{d}\mu = \sum_{n=1}^{N}c_{0,n}+\int \sum_{n=1}^{N}\sum_{\ell\in [-\lfloor N^{ \rho}\rfloor ,\lfloor N^{\rho}\rfloor]\setminus \{0\}}c_{\ell,n}e(\ell 2^nx)\, \mathrm{d}\mu+O\left(N^{1+2\tau-2\rho}\log N\right).$$
	Note that 
	\[1+2\tau-2\rho=1-\frac{\log2}{3\log3}.\]
	So to prove \eqref{eqsumfn1N}, it suffices to show that
	\begin{equation}
		\label{n1NellLF}
		\int \sum_{n=1}^{N}\sum_{\ell\in [-\lfloor N^{\rho}\rfloor ,\lfloor N^{\rho}\rfloor]\setminus \{0\}}c_{\ell,n}e(\ell2^nx)\, \mathrm{d}\mu=O\left(N^{1-\frac{\log2}{3\log3}}\log N\right).
	\end{equation}
	We have 
	\begin{equation}\label{eq1NNgamma}
		\int \sum_{n=1}^{N}\sum_{\ell\in [-\lfloor N^{\rho}\rfloor,\lfloor N^{\rho }\rfloor]\setminus \{0\}}c_{\ell,n}e(\ell 2^nx)\, \mathrm{d}\mu=\sum_{\ell\in [-\lfloor N^{ \rho}\rfloor,\lfloor N^{ \rho}\rfloor]\setminus \{0\}}\sum_{n=1}^{N}c_{\ell,n}\int e(\ell 2^nx)\, \mathrm{d}\mu.
	\end{equation} 
     Hence by Lemma~\ref{lem:weightedL1} and the fact that $|c_{\ell,n}|\ll n^{-\tau}$, 
	\begin{align*}
		\left|\sum_{n=1}^{N}c_{\ell,n}\int e(\ell 2^nx)\, \mathrm{d}\mu\right|\ll \sum_{n=1}^{N }n^{-\tau}|\widehat{\mu}(\ell 2^n)|
		\ll N^{1-\frac{\gamma}{2}-\tau}.
	\end{align*}
Plugging this upper bound into \eqref{eq1NNgamma} and noting that $1-\frac{\gamma}{2}-\tau+\rho=1-\frac{\log2}{3\log3}$, we obtain  \eqref{n1NellLF} and so complete the proof of the first part of the lemma. The second part follows from the first part, \eqref{cellnasy} and the fact that $\tau <\frac{\log2}{3\log3}$.  
\end{proof}

\begin{proof}[Proof of Lemma \ref{Lemma2}]
	Let $\rho$ be as in Lemma~\ref{Lemma1}. Multiplying out the bracket we have 
    \begin{equation}\label{Anna}
    \begin{aligned} 
		&\quad\int \left(\sum_{n=1}^{N}f_{n}(2^nx)-\int \sum_{n=1}^{N}f_{n}(2^ny)\,\mathrm{d} \mu(y) \right)^2\,\mathrm{d}\mu(x) \\
        &=\int\left( \sum_{n=1}^{N}f_{n}(2^nx)\right)^2 \mathrm{d}\mu-\left( \int \sum_{n=1}^{N}f_{n}(2^nx)\, \mathrm{d}\mu\right)^2.
	\end{aligned}
     \end{equation}
	To bound the first term on the right hand side of \eqref{Anna}, we use \eqref{errorbound} to get
	\begin{align}
		\label{Anna1}
		&\int \left( \sum_{n=1}^{N}f_{n}(2^nx)\right)^2\, \mathrm{d}\mu\nonumber \\
		=&\int \left(\sum_{n=1}^{N}c_{0,n}+\sum_{n=1}^{N}\sum_{\ell\in [-\lfloor N^{ \rho}\rfloor,\lfloor N^{ \rho}\rfloor]\setminus \{0\}}c_{\ell,n}e(\ell 2^nx)+O\left(N^{1-\frac{\log2}{3\log3}}\log N\right)\right)^{2}\, \mathrm{d}\mu.
	\end{align} 
Since
	\begin{equation}
		\label{trivial1}
		\sum_{n=1}^{N}c_{0,n}\asymp N^{1-\tau}
	\end{equation}
	and
	\begin{equation*}
\left|\sum_{n=1}^{N}\sum_{\ell\in [-\lfloor N^{\rho}\rfloor,\lfloor N^{ \rho}\rfloor]\setminus \{0\}}c_{\ell,n}e(\ell 2^nx)\right|\ll \sum_{\ell\in [-\lfloor N^{\rho}\rfloor,\lfloor N^{ \rho}\rfloor]\setminus \{0\}}\sum_{n=1}^{N}n^{-\tau}\ll N^{1- \tau+\rho},
	\end{equation*}
we expand the bracket in \eqref{Anna1} to obtain that
	\begin{align*}
		&\int \left( \sum_{n=1}^{N}f_{n}(2^nx)\right)^2\, \mathrm{d}\mu\\
		=&\underbrace{\left(\sum_{n=1}^{N}c_{0,n}\right)^{2}}_{A}+\underbrace{2\sum_{n=1}^{N}c_{0,n}\int \sum_{n=1}^{N}\sum_{\ell\in [-\lfloor N^{\rho}\rfloor,\lfloor N^{ \rho}\rfloor]\setminus \{0\}}c_{\ell,n}e(\ell 2^nx)\, \mathrm{d}\mu}_{B}\\
		&+\underbrace{\int \sum_{n=1}^{N}\sum_{\ell \in [-\lfloor N^{ \rho}\rfloor,\lfloor N^{ \rho}\rfloor]\setminus\{0\}}\sum_{m=1}^{N}\sum_{j\in [-\lfloor N^{\rho}\rfloor,\lfloor N^{ \rho}\rfloor]\setminus\{0\}}c_{\ell,n}c_{j,m}e((\ell 2^n+j2^m)x)\, \mathrm{d}\mu}_{C}\\
		&+O\left(N^{2-\tau-\frac{\log2}{3\log3}}\log N\right)+O\left(N^{2(1-\frac{\log2}{3\log3})}(\log N)^2\right)+O\left(N^{2-\tau+\rho-\frac{\log2}{3\log3}}\log N\right) .
	\end{align*} 
    These final three error terms are all $$O\left(N^{2-\tau+\rho-\frac{\log2}{3\log3}}\log N\right)=O\left(N^{2-\frac{\log2}{6\log3}}\log N\right).$$ Therefore it remains to consider terms $A$, $B$, and $C$. By Lemma \ref{Lemma1} we know that 
	$$\left(\int \sum_{n=1}^{N}f_{n}(2^nx)\, \mathrm{d}\mu\right)^{2} = \left(\sum_{n=1}^{N}c_{0,n}\right)^{2}+O\left(N^{2-\tau-\frac{\log2}{3\log3}}\log N\right).$$ Hence,
	$$\left(\sum_{n=1}^{N}c_{0,n}\right)^{2}-\left( \int \sum_{n=1}^{N}f_{n}(2^nx)\, \mathrm{d}\mu\right)^2=O\left(N^{2-\tau-\frac{\log2}{3\log3}}\log N\right)=O\left(N^{2-\frac{\log2}{6\log3}}\log N\right).$$
	Therefore, it remains to show that $B$ and $C$ are both $O\left(N^{2-\frac{\log2}{6\log3}}\log N\right)$. This we do below.
	
From the proof of \eqref{n1NellLF} in Lemma \ref{Lemma1}, we see that $$\left|\int \sum_{n=1}^{N}\sum_{\ell\in [-\lfloor N^{ \rho}\rfloor,\lfloor N^{\rho}\rfloor]\setminus \{0\}}c_{\ell,n}e(\ell 2^nx)\, \mathrm{d}\mu\right|=O\left(N^{1-\frac{\log2}{3\log3}}\right).$$ So by \eqref{trivial1},
	$$\left|\sum_{n=1}^{N}c_{0,n}\int \sum_{n=1}^{N}\sum_{\ell\in [-\lfloor N^{ \rho}\rfloor,\lfloor N^{\rho}\rfloor]\setminus \{0\}}c_{\ell,n}e(\ell 2^nx)\, \mathrm{d}\mu\right|=O\left(  N^{2-\tau-\frac{\log2}{3\log3}}\right).$$
	Therefore the term $B$ is $O\left(N^{2-\frac{\log2}{6\log3}}\log N\right).$

	 To bound the term $C$, we rewrite it as follows:
	\begin{align*}
		&\int \sum_{n=1}^{N}\sum_{\ell\in [-\lfloor N^{\rho}\rfloor,\lfloor N^{ \rho}\rfloor]\setminus\{0\}}\sum_{m=1}^{N}\sum_{j\in [-\lfloor N^{\rho}\rfloor,\lfloor N^{ \rho}\rfloor]\setminus\{0\}}c_{\ell,n}c_{j,m}e((\ell 2^n+j2^m)x)\, \mathrm{d}\mu\\
		=&\underbrace{\int\sum_{n=1}^{N}\sum_{\ell \in [-\lfloor N^{\rho}\rfloor,\lfloor N^{ \rho}\rfloor]\setminus\{0\}}\sum_{j\in [-\lfloor N^{\rho}\rfloor,\lfloor N^{ \rho}\rfloor]\setminus\{0\}}c_{\ell,n}c_{j,n}e((\ell2^n+j2^n)x)\, \mathrm{d}\mu}_{D}\\
		&+2\underbrace{\int \sum_{n=1}^{N-1}\sum_{\ell \in [-\lfloor N^{ \rho}\rfloor,\lfloor N^{ \rho}\rfloor]\setminus\{0\}}\sum_{m=n+1}^{N}\sum_{j\in [-\lfloor N^{\rho}\rfloor,\lfloor N^{ \rho}\rfloor]\setminus\{0\}}c_{\ell,n}c_{j,m}e((\ell 2^n+j2^m)x)\, \mathrm{d}\mu}_{E}.
	\end{align*}
By \eqref{cellnasy},  
	\begin{equation*}
	\begin{aligned}
\int\sum_{n=1}^{N}\sum_{\ell\in [-\lfloor N^{ \rho}\rfloor,\lfloor N^{\rho}\rfloor]\setminus\{0\}}\sum_{j\in [-\lfloor N^{\rho}\rfloor,\lfloor N^{\rho}\rfloor]\setminus\{0\}}c_{\ell,n}c_{j,n}e((\ell 2^n+j2^n)x)\, \mathrm{d}\mu=O\left(N^{1-2\tau+2\rho}\right). 
   \end{aligned}
     \end{equation*}
    Since $1-2\tau+2\rho=1+\frac{\log2}{3\log3}$, it follows that $D$ is $O\left(N^{2-\frac{\log2}{6\log3}}\log N\right)$. Below we bound term $E$.
    
We first bound $E$ from above by an expression to which we can apply Lemma~\ref{lem:weightedL1}:
	\begin{align*}
		&\left|\int \sum_{n=1}^{N-1}\sum_{\ell\in [-\lfloor N^{\rho}\rfloor,\lfloor N^{ \rho}\rfloor]\setminus\{0\}}\sum_{m=n+1}^{N}\sum_{j\in [-\lfloor N^{\rho}\rfloor,\lfloor N^{ \rho}\rfloor]\setminus\{0\}}c_{\ell,n}c_{j,m}e((\ell 2^n+j2^m)x)\, \mathrm{d}\mu\right| \\
		=&\left|\int \sum_{n=1}^{N-1}\sum_{\ell\in [-\lfloor N^{\rho}\rfloor,\lfloor N^{ \rho}\rfloor]\setminus\{0\}}\sum_{m=n+1}^{N}\sum_{j\in [-\lfloor N^{\rho}\rfloor,\lfloor N^{ \rho}\rfloor]\setminus\{0\}}c_{\ell,n}c_{j,m}e(2^{n}(\ell+j2^{m-n})x)\, \mathrm{d}\mu\right|\\
		\ll & \sum_{n=1}^{N-1}\sum_{\ell\in [-\lfloor N^{\rho}\rfloor,\lfloor N^{ \rho}\rfloor]\setminus\{0\}}\sum_{k=1}^{N-n}\sum_{j\in [-\lfloor N^{ \rho}\rfloor,\lfloor N^{ \rho}\rfloor]\setminus\{0\}}\frac{1}{n^{2\tau}}\left|\int e(2^{n}(\ell+j2^k)x)\, \mathrm{d}\mu\right|\\
		\ll & \sum_{n=1}^{N-1}\sum_{\ell\in [-\lfloor N^{\rho}\rfloor,\lfloor 
        N^{\rho}\rfloor]\setminus\{0\}}\sum_{k=1}^{N}\sum_{j\in [-\lfloor N^{\rho}\rfloor,\lfloor 
        N^{\rho}\rfloor]\setminus\{0\}}\frac{1}{n^{2\tau}}\left|\int e(2^{n}(\ell+j2^k)x)\, \mathrm{d}\mu\right|\\
		= & \sum_{k=1}^{N}\sum_{\ell\in [-\lfloor 
        N^{\rho}\rfloor,\lfloor 
        N^{\rho}\rfloor]\setminus\{0\}} \sum_{j\in [-\lfloor N^{\rho}\rfloor,\lfloor 
        N^{\rho}\rfloor]\setminus\{0\}} \sum_{n=1}^{N-1}\frac{1}{n^{2\tau}}\left|\int e(2^{n}(\ell+j2^k)x)\, \mathrm{d}\mu\right|.
	\end{align*}
	In the third line in the above we have used \eqref{cellnasy}. To proceed, we split the last term above according to $\ell+j2^k$ equals $0$ or not. Notice that if $k$ is such that there exist $\ell,j\in [-\lfloor 
    N^{\rho}\rfloor,\lfloor N^{\rho}\rfloor]\setminus\{0\}$ for which $\ell+j2^{k}=0,$ then $k\leq \lfloor\rho\log_{2}N \rfloor$. Thus we have
	\begin{align*}
		&\sum_{k=1}^{N}\sum_{\ell \in [-\lfloor N^{ \rho}\rfloor,\lfloor N^{ \rho}\rfloor]\setminus\{0\}}\sum_{j\in [-\lfloor N^{\rho}\rfloor,\lfloor N^{ \rho}\rfloor]\setminus\{0\}} \sum_{n=1}^{N-1}\frac{1}{n^{2\tau}}\left|\int e(2^{n}(\ell+j2^k)x)\, \mathrm{d}\mu\right|\\
 =&\underbrace{\sum_{k=1}^{\lfloor\rho\log_{2}N \rfloor}\sum_{\ell\in [-\lfloor N^{ \rho}\rfloor,\lfloor N^{ \rho}\rfloor]\setminus\{0\}}\sum_{j\in [-\lfloor N^{ \rho}\rfloor,\lfloor N^{ \rho}\rfloor]\setminus\{0\}} \sum_{n=1}^{N-1}\frac{1}{n^{2\tau}}\left|\int e(2^{n}(\ell+j2^k)x)\, \mathrm{d}\mu\right|}_{F}\\
		+& \underbrace{\sum_{k=\lfloor\rho \log_{2}N \rfloor+1}^{N}\sum_{\ell\in [-\lfloor N^{ \rho}\rfloor,\lfloor N^{ \rho}\rfloor]\setminus\{0\}}\sum_{j\in [-\lfloor N^{ \rho}\rfloor,\lfloor N^{ \rho}\rfloor]\setminus\{0\}} \sum_{n=1}^{N-1}\frac{1}{n^{2\tau}}\left|\int e(2^{n}(\ell+j2^k)x)\, \mathrm{d}\mu\right|}_{G}.
	\end{align*}
	To complete the proof, it suffices to show that $F$ and $G$ are both $O\left(N^{2-\frac{\log2}{3\log3}}\log N\right).$ 
    The term $F$ is $O\left(N^{2-\frac{\log2}{6\log3}}\log N\right)$, since
	\begin{align*}
		&\sum_{k=1}^{\lfloor\rho \log_{2}N \rfloor}\sum_{\ell\in [-\lfloor N^{ \rho}\rfloor,\lfloor N^{ \rho}\rfloor]\setminus\{0\}}\! \sum_{j\in [-\lfloor N^{ \rho}\rfloor,\lfloor N^{ \rho}\rfloor]\setminus\{0\}} \!\sum_{n=1}^{N-1}\frac{1}{n^{2 \tau}}\left|\int e(2^{n}(\ell+j2^k)x)\, \mathrm{d}\mu\right| \\
        &=O(N^{1-2\tau+2\rho}\log N)
     =O\left(N^{1+\frac{\log2}{3\log3}}\log N\right)=O\left(N^{2-\frac{\log2}{6\log3}}\log N\right).
	\end{align*}
	 Now we focus on $G$. If $\ell+j2^{k}\neq 0$, since $2\tau\leq\frac{\log2}{3\log3}<1-\frac{\gamma}{2}$, we can apply Lemma~\ref{lem:weightedL1} to yield that
	\begin{align*}
		\sum_{n=1}^{N-1}\frac{1}{n^{2\tau}}\left|\int e(2^{n}(l+j2^k)x)\, \mathrm{d}\mu\right|\ll N^{1-\frac{\gamma}{2}-2\tau}
	\end{align*}
	Therefore, 
	\begin{align*}
		&\sum_{k=\lfloor\rho\log_{2}N \rfloor+1}^{N}\sum_{\ell\in [-\lfloor N^{ \rho}\rfloor,\lfloor N^{ \rho}\rfloor]\setminus\{0\}}\sum_{j\in [-\lfloor N^{\rho} \rfloor,\lfloor N^{ \rho}\rfloor]\setminus\{0\}} \sum_{n=1}^{N-1}\frac{1}{n^{2\tau}}\left|\int e(2^{n}(\ell+j2^k)x)\, \mathrm{d}\mu\right| \\
		&\ll N^{1+\rho+\rho+1-\frac{\gamma}{2}-2\tau}
		=N^{2-\frac{\log2}{6\log3}}.
	\end{align*}
	This completes the proof of the lemma.  
\end{proof}

\section{Proof of Theorem \ref{Mainthm}}\label{Sec:Mainthm}

Equipped with Lemmas~\ref{Lemma1} and  \ref{Lemma2} we are now in a position to prove Theorem~\ref{Mainthm}. We follow the strategy in \cite[Section 2.3]{Baker25A}. 

\begin{proof}[Proof of Theorem \ref{Mainthm}]
	For convenience, write 
	\[\tau_0=\frac{\log2}{12\log3},\qquad \theta=2-\frac{\log2}{6\log3}.\]
	Recall that $\theta$ is the exponent appearing in Lemma~\ref{Lemma2}. Fix $0<\tau<\tau_0$. 
    Since $\frac{\theta}{2}<1-\tau$,
	we can find a large positive integer $Q$ such that 
	\[\frac{Q\theta+1}{2}<M<Q(1-\tau)\]
	for some positive integer $M$. Let such $Q$ and $M$ be fixed.

	Let $\epsilon>0$ be arbitrary and $(f_n)_{n=1}^{\infty}$ be a sequence of $C^2$ functions satisfying properties $(1)-(6)$ for $c_{1}=1,$ $c_{2}=1+\epsilon$ and $(x_{n})_{n=1}^{\infty}$. It follows from these properties and this choice of parameters that for any $x\in C$,
	\begin{equation}
		\label{Counting upper bound}
		\#\{1\leq n\leq N:\|2^{n}x-x_n\|< n^{-\tau}\}\leq \sum_{n=1}^{N}f_{n}(2^{n}x).
	\end{equation}
	By our choice of $c_{1}$ and $c_{2}$ it is also clear that 
	\begin{equation}
		\label{abequation}
		\sum_{n=1}^{N}c_{0,n}\leq 2(1+\epsilon)\sum_{n=1}^{N}n^{-\tau}.
	\end{equation}
	By Lemma \ref{Lemma2} and Markov's inequality, we have the following for any $N\in\mathbb{N}$,
	\begin{align*}
		&\quad\mu\left(x\in C:\left|\sum_{n=1}^{N^{Q}}f_{n}(2^{n}x)-\int \sum_{n=1}^{N^{Q}}f_{n}(2^ny)\, \mathrm{d}\mu(y)\right|\geq N^{M}\right)\\
		&=\mu\left(x\in C:\left(\sum_{n=1}^{N^{Q}}f_{n}(2^{n}x)-\int \sum_{n=1}^{N^{Q}}f_{n}(2^ny)\, \mathrm{d}\mu(y)\right)^{2}\geq N^{2M}\right)\\
		&\ll \frac{N^{ Q\theta}\log N}{N^{2M}}
		=\frac{\log N}{N^{2M-Q\theta}}.
	\end{align*}
	Since $2M-Q\theta>1$, it follows that
	$$\sum_{N=1}^{\infty}\mu\left(x\in C:\left|\sum_{n=1}^{N^{Q}}f_{n}(2^{n}x)-\int \sum_{n=1}^{N^{Q}}f_{n}(2^ny)\, \mathrm{d}\mu(y)\right|\geq N^{M}\right)<\infty.$$
	Therefore by the  Borel-Cantelli lemma, for $\mu$-a.e. $x\in C$, we have 
	\begin{equation}
		\label{eventually}
		\left|\sum_{n=1}^{N^{Q}}f_{n}(2^{n}x)-\int \sum_{n=1}^{N^{Q}}f_{n}(2^ny)\, d\mu(y)\right|<N^{M}
	\end{equation} 
	for all large $N$. For $N\in\mathbb{N}$, we let $K_{N}\in\mathbb{N}$ such that
    \begin{equation}\label{defKN}
    K_{N}^{Q}\leq N< (K_{N}+1)^{Q}. 
    \end{equation}
	  Notice that
	\begin{equation}
		\label{Kasymptotics}\lim_{K\to\infty}\frac{\sum_{n=1}^{(K+1)^{Q}}n^{-\tau}}{\sum_{n=1}^{K^{Q}}n^{-\tau}}=1\quad {\rm and }\quad \sum_{n=1}^{K^{Q}}n^{-\tau}\asymp K^{(1-\tau)Q}.
	\end{equation}

	Now let $x$ belong to the full $\mu$ measure set for which \eqref{eventually} holds for all sufficiently large $N$. Then we have
	\begin{align*}
		&\limsup_{N\to\infty}\frac{\#\{1\leq n\leq N:\|2^{n}x-x_n\|< n^{-\tau}\}}{2\sum_{n=1}^{N}n^{-\tau}}\\
		\stackrel{\eqref{Counting upper bound}}{\leq} & \limsup_{N\to\infty}\frac{\sum_{n=1}^{N}f_{n}(2^nx)}{2\sum_{n=1}^{N}n^{-\tau}}\\
		\stackrel{\eqref{defKN}}\leq &  \limsup_{N\to\infty}\frac{\sum_{n=1}^{(K_{N}+1)^{Q}}f_{n}(2^nx)}{2\sum_{n=1}^{K_{N}^{Q}}n^{-\tau}}\\
		\stackrel{\eqref{eventually}}{\leq}& \limsup_{N\to\infty}\frac{\int \sum_{n=1}^{(K_{N}+1)^{Q}}f_{n}(2^ny)\, \mathrm{d}\mu(y)+(K_{N}+1)^{M}}{2\sum_{n=1}^{K_{N}^{Q}}n^{-\tau}}\\
		\stackrel{\text{Lemma }\ref{Lemma1}}{=}&\limsup_{N\to\infty}\frac{\sum_{n=1}^{(K_{N}+1)^{Q}}c_{0,n}+(K_{N}+1)^{M}+O\left(K_{N}^{(1-\frac{\log2}{3\log3})Q}\log K_{N}\right)}{2\sum_{n=1}^{K_{N}^{Q}}n^{-\tau}}\\
		\stackrel{\eqref{abequation}}{\leq}& \limsup_{N\to\infty}\frac{2(1+\epsilon)\sum_{n=1}^{(K_N+1)^{Q}}n^{-\tau}+(K_{N}+1)^{M}+O\left(K_{N}^{(1-\frac{\log2}{3\log3})Q}\log K_{N}\right)}{2\sum_{n=1}^{K_{N}^{Q}}n^{-\tau}}\\
		\stackrel{\eqref{Kasymptotics}}{=}&1+\epsilon.
	\end{align*}
	 It follows that for $\mu$-a.e. $x$, we have  
	$$\limsup_{N\to\infty}\frac{\#\{1\leq n\leq N:\|2^{n}x-x_n\|<n^{-\tau}\}}{2\sum_{n=1}^{N}n^{-\tau}}\leq 1+\epsilon.$$ Since $\epsilon$ is arbitrary, we conclude that for $\mu$-a.e. $x$,
	\begin{equation}
		\label{limsup}
		\limsup_{N\to\infty}\frac{\#\{1\leq n\leq N:\|2^{n}x-x_n\|<n^{-\tau}\}}{2\sum_{n=1}^{N}n^{-\tau}}\leq 1.
	\end{equation}

	By an analogous argument, this time taking a sequence of $C^{2}$ functions $(f_n)_{n=1}^{\infty}$ satisfying properties $(1)-(6)$ for $c_{1}=1-\epsilon$, $c_{2}=1$ and $(x_{n})_{n=1}^{\infty}$, we can show that for $\mu$-a.e. $x$, 
$$\liminf_{N\to\infty}\frac{\#\{1\leq n\leq N:\| 2^{n}x-x_n\|<n^{-\tau}\}}{2\sum_{n=1}^{N}n^{-\tau}}\geq 1-\epsilon.$$ 
Again by the arbitraness of  $\epsilon$, we have that for $\mu$-a.e. $x$,
\begin{equation*}
\liminf_{N\to\infty}\frac{\#\{1\leq n\leq N:\|2^{n}x-x_n\|< n^{-\tau}\}}{2\sum_{n=1}^{N}n^{-\tau}}\geq 1.
\end{equation*}
This combining with \eqref{limsup} completes the proof. 
\end{proof}

\section{Proof of Theorem~\ref{thmconverge}}\label{sec:thmconverge}

Recall that $\mu$  is the Cantor-Lebesgue measure and $\gamma=\frac{\log 2}{\log 3}$. 

We prove Theorem \ref{thmconverge} by combining Lemma~\ref{LemL2-dyadic} with a modification of the method from \cite[Theorem 2]{ABCY23A}.
\begin{proof}[Proof of Theorem~\ref{thmconverge}]
Recall that our aim is to show that for any  $\tau>\frac{1}{\gamma}-\frac{1-\gamma}{3-\gamma}$, \(\mu (W_{2}(\psi_{\tau},y))= 0\). Fix such a  $\tau$. For \(n\in  \mathbb{N}\), put
	\[
	\sigma_{n} = n^{-\tau},\qquad \delta_{n} = n^{-\varrho},
	\]
	where $0<\varrho<\min(\tau,\frac{\gamma}{2})$ is to be determined. Let \(N\in \mathbb{N}\) be large such that $N^{\tau-\varrho}\geq150$.

	Write \(G_{N}\) for the set of integers \(n\in [N,2N]\) such that
	\[
	\sum_{1\leq |h|\leq 2/\delta_{2N}}|\widehat{\mu} (h2^{n})|\leq1
	\]
	and let \(B_{N}\) be its complement in \([N,2N]\cap \mathbb{Z}\). Then we have \begin{equation}\label{eqBnC}
		\sum_{n\in B_{N}}\sum_{1\leq |h|\leq 2/\delta_{2N}}|\widehat{\mu} (h2^{n})|\geq\sum_{n\in B_{N}}1=\#B_{N}.
	\end{equation}
On the other hand, it follows from Cauchy-Schwarz inequality and Lemma~\ref{LemL2-dyadic} that
	\begin{align*}
		\sum_{n\in B_{N}}\sum_{1\leq |h|\leq 2/\delta_{2N}}|\widehat{\mu} (h2^{n})|&=\sum_{1\leq |h|\leq 2/\delta_{2N}}\sum_{n\in B_{N}}|\widehat{\mu}(h2^{n})|\\
        &\leq\sum_{1\leq |h|\leq 2/\delta_{2N}}(\#B_{N})^{\frac{1}{2}}\left(\sum_{n\in B_{N}}|\widehat{\mu}(h2^{n})|^{2}\right)^{\frac{1}{2}} \\
		&\leq\sum_{1\leq |h|\leq 2/\delta_{2N}}(\#B_{N})^{\frac{1}{2}}\left(\sum_{n=N}^{2N}|\widehat{\mu}(h2^{n})|^{2}\right)^{\frac{1}{2}} \\
&\ll(\#B_{N})^{\frac{1}{2}}\sum_{1\leq |h|\leq 2/\delta_{2N}}N^{\frac{1-\gamma}{2}} \\ &\ll(\#B_{N})^{\frac{1}{2}}N^{\varrho+\frac{1-\gamma}{2}}.
	\end{align*}
	This together with \eqref{eqBnC} gives that
	$$ \#B_{N}\ll N^{2\varrho+1-\gamma}.$$

For \(n\in \mathbb{N}\) and \(\sigma >0\), denote
	\[
	A_{n}^{y}(\sigma) = \{x\in [0,1]:\| 2^{n}x - y\| < \sigma\}.
	\]
Then 
	\[
	W_{2}(\psi_{\tau},y) = \limsup_{n\to\infty} A_{n}^{y}(\sigma_{n}).
	\]
Hence by the Borel-Cantelli lemma, to prove \(\mu (W_{2}(\psi_{\tau},y))= 0\), it suffices to show that
	\[
	\sum_{n = 1}^{\infty}\mu (A_{n}^{y}(\sigma_{n}))<+\infty.
	\]
	
We need the following estimate  from \cite[Section 2.1]{ACY24}:
\begin{equation}\label{Ahlforsregular}
\mu (A_{n}^{y}(\sigma_{n}))\ll \sigma_{n}^{\gamma}  \qquad (n\in \N).
\end{equation}
Moreover, by \cite[Theorem 4.1]{Yu21R}, we have
\begin{equation}\label{eqYuEst}
\mu (A_{n}^{y}(\delta_{n}))\ll \delta_{n}\left(1 + \sum_{1\leqslant |h|\leqslant 2 / \delta_{n}}|\widehat{\mu} (h2^{n})|\right)\qquad (n\in \mathbb{N}).
\end{equation}
Since for  $n\in G_N$,
	\[\sum_{1\leq |h|\leq 2/\delta_n}|\widehat{\mu} (h2^{n})|\leq\sum_{1\leq |h|\leq 2/\delta_{2N}}|\widehat{\mu} (h2^{n})|\leq1, \]
we see from \eqref{eqYuEst} that 
\begin{equation}\label{Good}
		\mu (A_{n}^{y}(\delta_{n}))\ll \delta_n=n^{-\varrho}, \qquad \forall n\in G_N.
	\end{equation}
It was also proved in \cite{ABCY23A} that for all $n\in[N,2N]\cap\mathbb{Z}$, 
\begin{equation}\label{eqsigmadeltan}
\mu (A_{n}^{y}(\sigma_{n}))\ll
	\frac{(\sigma_{n} / 2^{n})^{\gamma}}{(\delta_{n} / 2^{n})^{\gamma}}\,
	\mu (A_{n}^{y}(\delta_{n}))=\sigma_{n}^{\gamma}\delta_{n}^{- \gamma}\mu (A_{n}^{y}(\delta_{n}))=n^{-\gamma(\tau-\varrho)}\mu (A_{n}^{y}(\delta_{n})).
\end{equation}

Now, combing the above estimates, we have 
\[
	\begin{aligned}
		\sum_{n = N}^{2N}\mu (A_{n}^{y}(\sigma_{n}))
		=&\sum_{n\in G_N}\mu (A_{n}^{y}(\sigma_{n}))+\sum_{n\in B_N}\mu (A_{n}^{y}(\sigma_{n})) \\
\stackrel{\eqref{Ahlforsregular}}{\ll}&\sum_{n\in G_N}\mu (A_{n}^{y}(\sigma_{n}))+\sum_{n\in B_N}\sigma_n^{\gamma}\\
\stackrel{\eqref{eqsigmadeltan}}{\ll} &\sum_{n\in G_N}n^{-\gamma(\tau-\varrho)}\mu (A_{n}^{y}(\delta_{n}))
		+ \sum_{n\in B_N}\sigma_{N}^{\gamma}\\
\stackrel{\eqref{Good}}{\ll} & \sum_{n = N}^{2N}n^{-\gamma(\tau-\varrho)-\varrho}
		+ N^{2\varrho+1-\gamma-\tau\gamma}\\
		\ll &\  N^{1-\gamma(\tau-\varrho)-  \varrho}+N^{2\varrho+1-\gamma-\tau\gamma}.
	\end{aligned}
	\]
	Write
	\[
	\sum_{n = 1}^{\infty}\mu (A_{n}^{y}(\sigma_{n}))
	\leqslant \sum_{k = 0}^{\infty}\sum_{n = 2^{k}}^{2^{k + 1}}\mu (A_{n}^{y}(\sigma_{n}))\ll\sum_{k = 0}^{\infty}(2^{k(1-\gamma(\tau-\varrho)- \varrho)}+2^{k(2\varrho+1-\gamma-\tau\gamma)}),
	\]
	which converges when
	\[1-\gamma(\tau-\varrho)-\varrho<0\quad \text{ and } \quad  2\varrho+1-\gamma-\tau\gamma<0.\]
	Equivalently,
	\[\tau>\max\left\{\frac{1-(1-\gamma)\varrho}{\gamma}, \frac{2\varrho+1-\gamma}{\gamma}\right\}.\]
The minimum of the right-hand side is attained when
	\[\frac{1-(1-\gamma)\varrho}{\gamma}=\frac{2\varrho+1-\gamma}{\gamma},\]
	which gives
	\[\varrho =\frac{\gamma}{3-\gamma} \quad \text{ and } \quad \frac{2\varrho+1-\gamma}{\gamma}=\frac{1}{\gamma}-\frac{1-\gamma}{3-\gamma}.\]
	Note that $\varrho<\frac{\gamma}{2}$ and 
$\varrho<\frac{2\varrho+1-\gamma}{\gamma}$. Therefore, letting $\varrho =\frac{\gamma}{3-\gamma}$ we see from the above proof that  $\mu(W_{2}(\psi_{\tau},y))=0$ whenever 
\[\tau>\frac{1}{\gamma}-\frac{1-\gamma}{3-\gamma}\approx 1.4292.\]
This completes the proof.
\end{proof}

\section{Key Lemmas for Theorem~\ref{thmsselfsimilar}}\label{Sec:KL}
The general strategy to prove Theorem \ref{thmsselfsimilar} is similar to that of Theorem~\ref{Mainthm}. However, some modifications of key lemmas are necessary, which we detail in this section.

Recall that $\mu_{\mathbf{p}}$ is the self-similar measure on $K_{b,D}$ satisfying 
\[\mu_{\mathbf{p}}=\sum_{j=1}^{\#D}p_j\mu_{\mathbf{p}}\circ\phi_j^{-1},\]
where 
\[\phi_j(x)=\frac{x+k_j}{b}, \quad D=\{k_1,k_2,\ldots, k_{\#D}\}\ {\rm and}\  k_1<k_2<\cdots<k_{\#D}.\]
It is well-known that (see \cite[page 342]{Strichartz1993})
\begin{equation*}
\widehat\mu_{\mathbf{p}}(x)
	= \prod_{q=1}^{\infty}\sum_{j=1}^{\#D}p_{j}e\left(-\frac{k_{j}x}{b^{q}}\right),
	\qquad
|\widehat\mu_{\mathbf{p}}(x)|
	= \prod_{q=1}^{\infty}\left|\sum_{j=1}^{\#D}p_{j}e\left(-\frac{k_{j}x}{b^{q}}\right)\right|.
\end{equation*}
Similar to Theorem~\ref{Mainthm}, 
the key to prove Theorem \ref{thmsselfsimilar} is to estimate the following weighted sum
\begin{equation*}
\sum_{n=1}^{N}n^{-\sigma}|\widehat{\mu}_{\mathbf{p}}(ht^{n})|,
\end{equation*}
where  $\sigma>0$, $h\in\mathbb{Z}\setminus\{0\}$, and $t\geq2$ is an integer coprime to $b$; see Lemma~\ref{Leweightedsum}. 

We first give an elementary lemma, which is an extension of Lemma \ref{lemaverageCOS2}.
\begin{lem}\label{Keyel}
	Suppose $b$ is a prime number, $\ell\in\mathbb{N}$ and $u\in\mathbb{Z}$. Then for every $r\in\mathbb{N}$,
	\begin{equation}\label{I1}
		\frac{1}{b^{r}}\sum_{y=b^{\ell}s+u,\, 0\leq s<b^{r}}\prod_{q=\ell+1}^{\ell+r}\left|\sum_{j=1}^{\#D}p_{j}e\left(-\frac{k_{j}y}{b^{q}}\right)\right|^{2}\leq\left(\min_{1\leq j_{1}<j_{2}\leq\#D}(1-2p_{j_{1}}p_{j_{2}})\right)^{r}.
	\end{equation}
	In particular, if $\#D=2$, then
\begin{equation}\label{E1}
		\frac{1}{b^{r}}\sum_{y=b^{\ell}s+u,\, 0\leq s<b^{r}}\prod_{q=\ell+1}^{\ell+r}\left|\sum_{j=1}^{2}p_{j}e\left(-\frac{k_{j}y}{b^{q}}\right)\right|^{2}=(1-2p_{1}p_{2})^{r}.
	\end{equation}
\end{lem}
\begin{proof} 
Note that to prove \eqref{I1}, it is equivalent to show that  for every pair $(j_{1},j_{2})$ with $1\leq j_{1}<j_{2}\leq\#D$,
	$$\frac{1}{b^{r}}\sum_{y=b^{\ell}s+u, 0\leq s\leq b^{r}-1}\prod_{q=\ell+1}^{\ell+r}\left|\sum_{j=1}^{\#D}p_{j}e\left(-\frac{k_{j}y}{b^{q}}\right)\right|^{2}\leq\left(1-2p_{j_{1}}p_{j_{2}}\right)^{r}.$$
	Fix $(j_{1},j_{2})$ satisfying $1\leq j_{1}<j_{2}\leq\#D$. Since
	\begin{equation*}
    \begin{aligned}
&\quad\left|\sum_{j=1}^{\#D}p_{j}e\left(-\frac{k_{j}y}{b^{q}}\right)\right| \\
&=\left|p_{j_{1}}e\left(-\frac{k_{j_{1}}y}{b^{q}}\right)+p_{j_{2}}e\left(-\frac{k_{j_{2}}y}{b^{q}}\right)
		+\sum_{1\leq j\leq \#D,j\notin\{j_{1},j_{2}\}}p_{j}e\left(-\frac{k_{j}y}{b^{q}}\right)\right|,
    \end{aligned}
	\end{equation*}
	it follows from the triangle inequality that
	\begin{equation*}
		\begin{aligned}
			\left|\sum_{j=1}^{\#D}p_{j}e\left(-\frac{k_{j}y}{b^{q}}\right)\right|&\leq\left|p_{j_{1}}e\left(-\frac{k_{j_{1}}y}{vb^{q}}\right)+p_{j_{2}}e\left(-\frac{k_{j_{2}}y}{b^{q}}\right)\right|
			+\left|\sum_{1\leq j\leq\#D,j\notin\{j_{1},j_{2}\}}p_{j}e\left(-\frac{k_{j}y}{b^{q}}\right)\right| \\
			&\leq\left|p_{j_{1}}+p_{j_{2}}e\left(-\frac{(k_{j_{2}}-k_{j_{1}})y}{b^{q}}\right)\right|+1-p_{j_{1}}-p_{j_{2}} \\
			&=\sqrt{(p_{j_{1}}-p_{j_{2}})^{2}+4p_{j_{1}}p_{j_{2}}\cos^{2}\left(\frac{\pi(k_{j_{2}}-k_{j_{1}})y}{b^{q}}\right)}+1-p_{j_{1}}-p_{j_{2}}.
		\end{aligned}
	\end{equation*}
	Hence,
	\begin{equation*}
		\begin{aligned}
			&\quad\left|\sum_{j=1}^{\#D}p_{j}e\left(-\frac{k_{j}y}{b^{q}}\right)\right|^{2} \\
			&\leq(p_{j_{1}}-p_{j_{2}})^{2}+4p_{j_{1}}p_{j_{2}}\cos^{2}\left(\frac{\pi(k_{j_{2}}-k_{j_{1}})y}{b^{q}}\right)+(1-p_{j_{1}}-p_{j_{2}})^{2} \\
			&\quad +2(1-p_{j_{1}}-p_{j_{2}})\sqrt{(p_{j_{1}}-p_{j_{2}})^{2}+4p_{j_{1}}p_{j_{2}}\cos^{2}\left(\frac{\pi(k_{j_{2}}-k_{j_{1}})y}{b^{q}}\right)} \\
			&\leq(p_{j_{1}}-p_{j_{2}})^{2}+4p_{j_{1}}p_{j_{2}}\cos^{2}\left(\frac{\pi(k_{j_{2}}-k_{j_{1}})y}{b^{q}}\right)+(1-p_{j_{1}}-p_{j_{2}})^{2} \\
			&\quad +2(1-p_{j_{1}}-p_{j_{2}})(p_{j_{1}}+p_{j_{2}}) \\
			&=(p_{j_{1}}-p_{j_{2}})^{2}+(1-p_{j_{1}}-p_{j_{2}})(1+p_{j_{1}}+p_{j_{2}})+4p_{j_{1}}p_{j_{2}}\cos^{2}\left(\frac{\pi(k_{j_{2}}-k_{j_{1}})y}{b^{q}}\right)\\
        &=1-4p_{j_{1}}p_{j_{2}}+4p_{j_{1}}p_{j_{2}}\cos^{2}\left(\frac{\pi(k_{j_{2}}-k_{j_{1}})y}{b^{q}}\right).
		\end{aligned}
	\end{equation*}
Using
	$$\cos^{2}(\pi x)=\frac{1}{2}+\frac{1}{4}e(x)+\frac{1}{4}e(-x),\quad \forall x\in\mathbb{R},$$
	  we obtain that
\begin{equation*}
		\left|\sum_{j=1}^{\#D}p_{j}e\left(-\frac{k_{j}y}{b^{q}}\right)\right|^{2}\leq1-2p_{j_{1}}p_{j_{2}}+p_{j_{1}}p_{j_{2}}e\left(\frac{(k_{j_{2}}-k_{j_{1}})y}{b^{q}}\right)
		+p_{j_{1}}p_{j_{2}}e\left(\frac{(k_{j_{2}}-k_{j_{1}})y}{b^{q}}\right).
	\end{equation*}
It follows that
\begin{equation}\label{I3}
		\begin{aligned}
			&\quad\prod_{q=\ell+1}^{\ell+r}\left|\sum_{j=1}^{\#D}p_{j}e\left(-\frac{k_{j}y}{b^{q}}\right)\right|^{2} \\
			&\leq\prod_{q=\ell+1}^{\ell+r}\left(1-2p_{j_{1}}p_{j_{2}}+p_{j_{1}}p_{j_{2}}e\left(\frac{(k_{j_{2}}-k_{j_{1}})y}{b^{q}}\right)
			+p_{j_{1}}p_{j_{2}}e\left(-\frac{(k_{j_{2}}-k_{j_{1}})y}{b^{q}}\right)\right).
		\end{aligned}
	\end{equation}
	Denote
	$$c_{0}=1-2p_{j_{1}}p_{j_{2}},\quad c_{-1}=c_{1}=p_{j_{1}}p_{j_{2}}.$$
Expanding the right side of \eqref{I3}, a typical term is of the form
$$c_{\boldsymbol{\varepsilon}}e\left((k_{j_{2}}-k_{j_{1}})y \sum_{q=\ell+1}^{\ell+r}\frac{\varepsilon_{q}}{b^{q}}\right),$$
	where
$$\boldsymbol{\varepsilon}=(\varepsilon_{\ell+1},\varepsilon_{\ell+2},\ldots,\varepsilon_{\ell+r})\in\{-1,0,1\}^{r}\quad {\rm and}\quad c_{\boldsymbol{\varepsilon}}=\prod_{q=\ell+1}^{\ell+r}c_{\varepsilon_{q}}.$$
Hence the right side of \eqref{I3} is equal to
\begin{equation*}
\sum_{\boldsymbol{\varepsilon}\in\{-1,0,1\}^{r}}c_{\boldsymbol{\varepsilon}}e\left( (k_{j_{2}}-k_{j_{1}})y \sum_{q=\ell+1}^{\ell+r}\frac{\varepsilon_{q}}{b^{q}}\right).
	\end{equation*}
	Therefore,
	\begin{equation*}
		\begin{aligned}
			&\quad\sum_{y=b^{\ell}s+u,\, 0\leq s< b^{r}}\prod_{q=\ell+1}^{\ell+r} \left(1-2p_{j_{1}}p_{j_{2}}+p_{j_{1}}p_{j_{2}}e\left(\frac{(k_{j_{2}}-k_{j_{1}})y}{b^{q}}\right)
			+p_{j_{1}}p_{j_{2}}e\left(-\frac{(k_{j_{2}}-k_{j_{1}})y}{b^{q}}\right)\right) \\
&=\sum_{y=b^{\ell}s+u,\, 0\leq s<b^{r}}\sum_{\boldsymbol{\varepsilon}\in\{-1,0,1\}^{r}}c_{\boldsymbol{\varepsilon}}e\left( (k_{j_{2}}-k_{j_{1}})y \sum_{q=\ell+1}^{\ell+r}\frac{\varepsilon_{q}}{b^{q}}\right)\\
&=\sum_{\boldsymbol{\varepsilon}\in\{-1,0,1\}^{r}}\sum_{y=b^{\ell}s+u,\, 0\leq s< b^{r}}c_{\boldsymbol{\varepsilon}}e\left( (k_{j_{2}}-k_{j_{1}})y \sum_{q=\ell+1}^{\ell+r}\frac{\varepsilon_{q}}{b^{q}}\right) \\
			&=\sum_{\boldsymbol{\varepsilon}\in\{-1,0,1\}^{r}}\left(c_{\boldsymbol{\varepsilon}}e\left( (k_{j_{2}}-k_{j_{1}})u \sum_{q=\ell+1}^{\ell+r}\frac{\varepsilon_{q}}{b^{q}}\right)
			\sum_{s=0}^{b^{r}-1}e\left((k_{j_{2}}-k_{j_{1}})s\frac{\sum_{q=\ell+1}^{\ell+r}\varepsilon_{q}b^{r+\ell-q}}{b^{r}}\right)\right).
		\end{aligned}
	\end{equation*}
	Since $b$ is a prime number and $\boldsymbol{\varepsilon}\in\{-1,0,1\}^{r}$,  an argument analogous to the proof of Lemma \ref{lemaverageCOS2} yields that 
	\begin{equation*}
		\begin{aligned}
			\sum_{s=0}^{b^{r}-1}e\left((k_{j_{2}}-k_{j_{1}})s\frac{\sum_{q=\ell+1}^{\ell+r}\varepsilon_{q}b^{r+\ell-q}}{b^{r}}\right)=\begin{cases} b^{r},\ &{\rm if}\ \boldsymbol{\varepsilon}=\boldsymbol{0}, \\ 0,\ &{\rm if}\ \boldsymbol{\varepsilon}\in\{-1,0,1\}^{r}\setminus\{\boldsymbol{0}\}. \end{cases}
		\end{aligned}
	\end{equation*}
Consequently, we have 
\begin{equation}\label{E2}
		\begin{aligned}
			&\sum_{y=b^{\ell}s+u, 0\leq s\leq b^{r}-1}\prod_{q=\ell+1}^{\ell+r} \left(1-2p_{j_{1}}p_{j_{2}}+p_{j_{1}}p_{j_{2}}e\left(\frac{(k_{j_{2}}-k_{j_{1}})y}{b^{q}}\right)
			+p_{j_{1}}p_{j_{2}}e\left(-\frac{(k_{j_{2}}-k_{j_{1}})y}{b^{q}}\right)\right) \\
			&\quad=(1-2p_{j_{1}}p_{j_{2}})^{r}b^{r}.
		\end{aligned}
	\end{equation}
	This together with \eqref{I3} gives \eqref{I1}. Finally, we prove the ``In particular'' part. Note that if $\#D=2$, then 
	\begin{equation*}
		\left|\sum_{j=1}^{ 2}p_{j}e\left(-\frac{k_{j}y}{b^{q}}\right)\right|^{2}=1-2p_{ 1}p_{2}+p_{1}p_{2}e\left(\frac{(k_{2}-k_{1})y}{b^{q}}\right)
		+p_{1}p_{2}e\left(\frac{(k_{2}-k_{1})y}{b^{q}}\right).
	\end{equation*}
Combing this with  \eqref{E2} gives \eqref{E1}.
\end{proof}
Since $b\geq3$ is a prime number and $t\geq2$ is an integer satisfying that $b\nmid t$,   the following lemma follows from the identical proof of Lemma~\ref{lemmod3lr} and Corollary~\ref{corhbnmod}.
\begin{lem}\label{lemodblr}
Denote $\ell=\nu_b(t^{b-1}-1)$. Let $h\in\Z\setminus\{0\}$ and write $
	h=b^m u$, where $m\geq 0$ and $u\in\Z\setminus\{0\}$ with $ \gcd(u,b)=1$. Then, for every integer $r\ge 0$,
\begin{equation*} 
	\left\{ht^{(b-1)n} \bmod b^{m+\ell+r}: 0\leq n<b^r\}
	=
	\{(h+b^{m+\ell}k)\bmod b^{m+\ell+r}: 0\leq k<b^{r}\right\}.
\end{equation*}
\end{lem}

Based on Lemmas \ref{Keyel} and \ref{lemodblr}, we can establish analogues to Lemmas~\ref{LemL2-dyadic} and \ref{lem:weightedL1}  as follows.

Let
\begin{equation}\label{beta}
\alpha:=1+\frac{\log\left(\min\limits_{1\leq j_{1}<j_{2}\leq\#D}(1-2p_{j_{1}}p_{j_{2}})\right)}{\log b}\quad \text{ and } \quad \beta:=\frac{1+\alpha}{2}. 
\end{equation}

\begin{lem}\label{L2}
Suppose $b\geq3$ is a prime number and $t\geq2$ is an integer satisfying that $b\nmid t$. Then there exists an absolute constant $c>0$, such that for every integer $a\geq0$, every positive integer $N$ and every non-zero integer $h$,
	$$\sum_{n=a}^{a+N-1}|\widehat{\mu}_{\mathbf{p}}(ht^{n})|^{2}\leq c^{2}N^{\alpha}.$$
	Consequently,
	$$\sum_{n=a}^{a+N-1}|\widehat{\mu}_{\mathbf{p}}(ht^{n})|\leq  cN^{\beta}.$$
	The constant $c$ is independent of $a$, $N$ and $h$.
\end{lem}

\begin{lem}\label{Leweightedsum}
	Suppose $b\geq3$ is a prime number and $t\geq2$ is an integer satisfying that $b\nmid t$. Let $\beta$ be defined in \eqref{beta} and 
    $\sigma\in(0,\beta)$. Then for every non-zero integer $h$,
	$$\sum_{n=1}^{N}n^{- \sigma}|\widehat{\mu}_{\mathbf{p}}(ht^{n})|\ll_{\sigma}N^{\beta-\sigma},$$
	where the implicit constant is independent of $N$ and $h$. 
\end{lem}
With Lemma \ref{Leweightedsum} in hand, we just need to repeat the proof of Theorem \ref{Mainthm} (in which we replace $\mu$ by $\mu_{\mathbf{p}}$ and $2$ by $t$) to establish Theorem \ref{thmsselfsimilar}, we omit the details.
 
{\noindent \bf  Acknowledgements}. The authors thank Professor Sanju Velani for constructive discussions, which are very helpful in improving the quality of the article. X.-R. Dai was supported by the National Key R\&D Program of China (No.2024YFA1013703), NSFC 12271534 and the Guangdong Province Key Laboratory of Computational Science at the Sun Yat-sen University. BL was supported by National Key R\&D Program of
China (No. 2024YFA1013700), NSFC 12271176 and Guangdong Natural Science Foundation 2024A1515010946. Y.-F. Wu was supported by the NSFC 12301110.


\begin{thebibliography}{100}

\bibitem{ABCY23A}
D. Allen, S. Baker, S. Chow, and H.  Yu.
A note on dyadic approximation in Cantor's set. Indag. Math. (N.S.) 34(1): 190–197, 2023.


\bibitem{ACY24} D. Allen, S. Chow, and H. Yu. Dyadic approximation in the middle-third Cantor set. Selecta Math. (N.S.) 29(1): Paper No. 11, 2023.

	
\bibitem{Baker25A}
S. Baker.
Approximating elements of the middle third Cantor set with dyadic rationals. 
Israel J. Math. 266(1):  285--305, 2025.

\bibitem{BHZ24K}
T. Bénard, W. K. He, and H. Zhang.
Khintchine dichotomy for self-similar measures. J. Amer. Math. Soc., 39(3): 587--623, 2026.



\bibitem{BD16M}
Y. Bugeaud and A. Durand.
Metric Diophantine approximation on the middle-third Cantor set.
{\em J. Eur. Math. Soc. (JEMS)} 18(6):1233--1272, 2016.


\bibitem{Cassels59}
J. W. S. Cassels.
On a problem of Steinhaus about normal numbers.
Colloq. Math. 7:95–101, 1959.

\bibitem{CLW26}
Y. Y. Chang, B. Li, and M. Wu.
On orbit complexity of dynamical systems: intermediate value property and level set related to a Furstenberg problem.
{\em Int. Math. Res. Not. IMRN}, 2026(11): rnag106, 2026.

\bibitem{CVY24C}
S. Chow, P. Varju, and  H. Yu.
Counting rationals and diophantine approximation in missing-digit Cantor sets. Adv. Math., 488, Paper No. 110807, 2026.

\bibitem{CU25S}
S. Chow and H. Yu. Simultaneous and multiplicative Diophantine approximation on missing-digit fractals. arXiv:2412.12070, 2024.


\bibitem{Falconer14}
K. J. Falconer. {\it Fractal geometry. Mathematical foundations and applications}. Third edition.
John Wiley \& Sons, Ltd., Chichester, 2014.

\bibitem{Furstenberg67}
H. Furstenberg.
Disjointness in ergodic theory, minimal sets, and a problem in Diophantine approximation.
{\em Math. Systems Theory}, 1:1--49,  1967.

\bibitem{HS15E}
M. Hochman and P. Shmerkin.
Equidistribution from fractal measures. 
{\em Invent. Math.}, 202(1):427--479, 2015.

\bibitem{Host95}
B. Host.
Nombres normaux, entropie, translations. 
{\em Israel J. Math.}, 91(1--3):419--428, 1995.

\bibitem{Jackson94T} D. Jackson. \textit{The theory of approximation,} American Mathematical Society Colloquium Publications, 11. American Mathematical Society, Providence, RI, 1994. viii+178 pp. ISBN: 0-8218-1011-1.


\bibitem{LSV07O}
J. Levesley, C. Salp, and S. L. Velani.
On a problem of K. Mahler: Diophantine approximation and Cantor sets. 
{\em Math. Ann.} 338(1):97--118, 2007.

\bibitem{LLW25Z}
B. Li, R. F. Li, and Y. F. Wu.
Zero-full law for well approximable sets in missing digit sets. 
{\em Math. Proc. Cambridge Philos. Soc.}, 178(1):81--102, 2025.

\bibitem{Mahler84}
K. Mahler. Some suggestions for further research. {\em  Bull. Austral. Math. Soc.}, 29(1):101--108, 1984.

\bibitem{Pongsriiam}
P. Pongsriiam.
{\em Analytic number theory for beginners.}
Second edition. Student Mathematical Library, 103. American Mathematical Society, Providence, RI, 2023.


\bibitem{Schleischitz21}
J. Schleischitz. On intrinsic and extrinsic rational approximation to Cantor sets. {\em Ergodic Theory Dynam. Systems}, 41(5):1560--1589, 2021.

\bibitem{Shmerkin19}
P. Shmerkin.
On Furstenberg's intersection conjecture, self-similar measures, and the $L^q$ norms of convolutions. 
{\em Ann. of Math. (2)},  189:319--391, 2019.

\bibitem{Strichartz1993}
R. S. Strichartz. Self-similar measures and their Fourier transforms, \Rmnum{2}, Trans. Amer.
Math. Soc. 336(1):335–361, 1993.

\bibitem{TWW24M}
B. Tan, B. W. Wang, and J. Wu.
Mahler's question for intrinsic Diophantine approximation on triadic Cantor set: the divergence theory. 
{\em Math. Z.}, 306(1), no.2, 2024.

\bibitem{Wu19A}
M. Wu.
A proof of Furstenberg's conjecture on the intersections of $\times p$- and $\times q$-invariant sets. 
{\em Ann. of Math. (2)},  189(3):707--751, 2019.

\bibitem{Yu21R} H. Yu. Rational points near self-similar sets. Preprint: arXiv:2101.05910, 2021.

\end{thebibliography}
\end{document}